\documentclass[11pt]{amsart}
\usepackage{amsmath}
\usepackage{a4wide}
\usepackage[utf8]{inputenc}
\usepackage{amssymb}
\usepackage{amsopn}
\usepackage{epsfig}
\usepackage{amsfonts}
\usepackage{latexsym}
\usepackage{amsthm}
\usepackage{enumerate}
\usepackage[UKenglish]{babel}
\usepackage{verbatim}
\usepackage{color}
\usepackage{calrsfs}
\usepackage{pgf}
\usepackage{tikz}
\usepackage{breqn}
\usetikzlibrary{arrows,automata}
\DeclareMathAlphabet{\pazocal}{OMS}{zplm}{m}{n}

\newtheorem{theorem}{Theorem}[section]
\newtheorem{lemma}[theorem]{Lemma}
\newtheorem{proposition}[theorem]{Proposition}
\newtheorem{corollary}[theorem]{Corollary}

\theoremstyle{definition}
\newtheorem{definition}[theorem]{Definition}
\newtheorem{example}[theorem]{Example}

\newtheorem{conjecture}{Conjecture}[section]
\newtheorem{question}[conjecture]{Question}
\newtheorem{main}{Theorem}

\theoremstyle{remark}
\newtheorem{remark}[theorem]{Remark}

\numberwithin{equation}{section}
\newcommand{\N}{\ensuremath{\mathbb{N}}}

\newcommand{\B}{\ensuremath{\mathcal{B}}}

\newcommand{\set}[1]{\left\{#1\right\}}
\renewcommand{\u}{\pazocal{U}}
\newcommand{\ub}{\mathbf{U}} 
\newcommand{\ul}{\ensuremath{\mathcal{U}}} 
\newcommand{\V}{\pazocal{V}}
\newcommand{\Tr}{\mathcal{T}}
\newcommand{\NTr}{\mathcal{NT}}
\newcommand{\vl}{\mathcal{V}}

\newcommand{\vb}{\mathbf{V}}
 
\newcommand{\I}{\mathcal{I}}
\newcommand{\IS}{\mathcal{I}^{*}}
\newcommand{\SI}{\mathcal{S}_{\mathcal{I}}}
\newcommand{\WI}{\mathcal{W}_{\mathcal{I}}}

\newcommand{\WB}{\mathbf{W}}

\newcommand{\C}{\mathcal{C}}
\newcommand{\Ci}{\mathcal{C}_1'}
\newcommand{\Cii}{\mathcal{C}_2'}
\newcommand{\Ciii}{\mathcal{C}_3'}
\newcommand{\Civ}{\mathcal{C}_4'}
\newcommand{\Cv}{\mathcal{C}_5'}
\newcommand{\om}{\omega}
\newcommand{\ga}{\gamma}

\newcommand{\f}{\infty}

\newcommand{\al}{\alpha}

\newcommand{\si}{\sigma}

\newcommand{\La}{\pazocal{L}}

\newcommand{\F}{\pazocal{F}}

\newcommand{\x}{\mathbf{x}}
\newcommand{\y}{\mathbf{y}}
\newcommand{\Sig}{\Sigma_{M}}

\newcommand{\lle}{\preccurlyeq}
\newcommand{\lge}{\succcurlyeq}

\title{On the specification property and synchronisation of unique $q$-expansions}
\date{\today}
\dedicatory{In loving memory of Patricia Barrera}
\author{Rafael Alcaraz Barrera}
\address{Instituto de F\'isica, Universidad Aut\'onoma de San Luis Potos\'i. Av. Manuel Nava 6, Zona Universitaria, C.P. 78290. San Luis Potos\'i, S.L.P. M\'exico}
\email{ralcaraz@ifisica.uaslp.mx}

\thanks{Research of R. Alcaraz Barrera was sponsored by CONACYT-FORDECYT 265667}

\begin{document}
\begin{abstract}
Given a positive integer $M$ and $q \in (1, M+1]$ we consider expansions in base $q$ for real numbers $x \in \left[0, {M}/{q-1}\right]$ over the alphabet $\set{0, \ldots, M}$. In particular, we study some dynamical properties of the natural occurring subshift $(\vb_q, \si)$ related to unique expansions in such base $q$. We characterise the set of $q \in \vl \subset (1,M+1]$ such that $(\vb_q, \si)$ has the specification property and the set of $q \in \vl$ such that $(\vb_q, \si)$ is a synchronised subshift. Such properties are studied by analysing the combinatorial and dynamical properties of the quasi-greedy expansion of $q$. We also calculate the size of such classes as subsets of $\vl$ giving similar results to those shown by Blanchard \cite{Bla1989} and Schmeling in \cite{Sch1997} in the context of $\beta$-transformations.    
\end{abstract}
\maketitle

\section{Introduction}
\label{sec:intro}

\noindent Since their introduction in R\'{e}nyi's \cite{Ren1957} and Parry's \cite{Par1960} seminal papers, the theory of expansions in non-integer bases, colloquially known as \emph{$\beta$-expansions} or \emph{$q$-expansions}, has received much attention by researchers in many areas of mathematics, most notably ergodic theory, fractal geometry, number theory and symbolic dynamics.

Let us remind the reader the setting of $q$-expansions. Let $M \in \N$ and set 
$$\Sig = \mathop{\prod}\limits_{i=1}^{\f}\set{0, 1, \ldots, M}$$
equipped with the product topology. Given $q\in(1,M+1]$ for every $x \in I_{q,M} = \left[0, M/(q-1)\right]$ there is a sequence $\x = x_1x_2\ldots \in \Sig$ satisfying
\begin{equation}
\label{eq:q-expansion}
x=\pi_{q}(\x):=\sum_{i=1}^\f x_i/q^i
\end{equation}
The sequence $\x$ is called an \textit{expansion of $x$ in base $q$} (or simply a \textit{$q$-expansion of $x$}).

The \emph{greedy $q$-shift} (most commonly known as \emph{$\beta$-shift}), emerges from the set of expansions generated by the greedy algorithm for $x \in [0,1]$. The greedy $q$-shift is the subshift $(\Sigma_q, \si)$ given by $$\Sigma_{q} = \set{\x \in \Sig : \si^n(\x) \lle \al(q) \textrm{ for every } n \geq 0}$$ where $\al(q)$ stands for the quasi-greedy expansion of $1$ in base $q$, that is, the lexicographically largest infinite $q$-expansion of $1$. The properties of $(\Sigma_q, \si)$ have been studied extensively. For example, it is widely known that the topological entropy of $(\Sigma_q, \si)$ is $\log (q)$. The following theorem summarises the results regarding the symbolic dynamics and the size of certain classes of $q$-shifts ---see \cite{BerMat1986}, \cite{FalPfi2009}, \cite{Sch1997} and \cite{Sid20031}.

\begin{theorem}\label{teo:Schmeling}
Let $q \in (1,M+1]$ and $(\Sigma_q, \si)$ be the corresponding greedy $q$-shift. Then: 
\begin{enumerate}[$i)$]
\item $(\Sigma_q, \si)$ is topologically mixing for every $q \in (1,M+1]$. 
\item $(\Sigma_q, \si)$ is a subshift of finite type if and only if $\al(q)$ is periodic. Moreover, $$\C_1 = \set{q \in (1,M+1] : (\Sigma_q, \si) \textrm{ is a subshift of finite type }}$$ is a countable and dense subset of $(1,M+1]$.
\item $(\Sigma_q, \si)$ is a sofic subshift if and only if $\al(q)$ is eventually periodic. Moreover, $$\C_2 = \set{q \in (1,M+1] : (\Sigma_q, \sigma) \textrm{ is a sofic subshift }}$$ is a countable subset. 
\item $(\Sigma_q, \si)$ has the specification property (see Definition \ref{def:topologicalproperties} $iii)$) if and only if $\al(q)$ does not contain arbitrarily long strings of consecutive $0$'s. Moreover $$\C_3 = \set{q \in (1,M+1] : (\Sigma_q, \sigma) \textrm{ has the specification property }}$$ is an uncountable subset of Lebesgue measure zero and $\dim_{\operatorname{H}}(\C_3) = 1$.
\item $(\Sigma_q, \si)$ is synchronised (see Definition \ref{def:topologicalproperties} $iv)$) if and only if the orbit of $\al(q)$ is not dense in $\Sigma_q$. Moreover $$\C_4 = \set{q \in (1,M+1]: \Sigma_q \textrm{ is synchronised }}$$ is a meagre set in $(1,M+1]$.
\item The set $$\C_5 = (1,M+1] \setminus \C_4$$ is a residual set. 
\end{enumerate}
\end{theorem}

Perhaps the most noticeable feature of expansions in non-integer bases is the fact that they are not always unique. In fact, for any $k\in \N \cup\{\aleph_{0}\}\cup\{2^{\aleph_{0}}\}$ and any $M\in \N$, there is $q\in(1,M+1]$ and $x\in I_{M,q}$ such that $x$ has precisely $k$ different $q$-expansions---see \cite{ErdJooKom1990, Sid2009}. Sidorov in \cite{Sid2003} describes the generic behaviour of the set of expansions; that is, for any $q\in(1,M+1)$ and for Lebesgue-almost-every $x\in I_{M,q}$ has a continuum of $q$-expansions. The situation described above is of course completely different to the usual expansions in integer bases where every number has a unique $M$-expansion except for a countable set of exceptions, the $M$-adic rationals, that have precisely two.

Another particularly well-studied topic in expansions in non-integer bases is the set of numbers in $I_{M,q}$ with a unique $q$-expansion. The properties of this set have received lots of attention recently, see for example \cite{AlcBakKon2016, All2017, Bak2014, DeVKom2008, GeTan2017, KalKonLiLu2017, KonLi2015} and references therein. Let us remind the reader of this setting. For $q\in(1, M+1]$ the  \emph{univoque set on base $q$} is $$\u_q:=\{x\in[0,M/(q-1)]: x \textrm{ has a unique }q\textrm{-expansion}\}.$$  We consider $$\ub_q:=\pi_{q}^{-1}(\u_{q}) \subset \Sig$$ to be the corresponding set of expansions. The pair $(\ub_q, \si)$ is not necessarily a subshift \cite[Theorem 1.8]{DeVKom2008}. However, the set
\begin{equation}\label{eq:symshift}
\vb_q:=\set{\x \in \Sig : \overline {\al(q)}\preccurlyeq \si^n(\x)\preccurlyeq \al(q)\textrm{ for every } n\geq 0}
\end{equation}
where $\overline{\al(q)} = (M-\al_i(q))$, is a closed, forward-$\si$-invariant and non-empty subset of $\Sig$ for every $q \in [q_G(M), M+1]$ where $q_G(M)$ is \emph{the generalised golden ratio} introduced by Baker in \cite{Bak2014} (see also \cite[Lemma 2.4]{AlcBakKon2016}). We will refer to $(\vb_q, \si)$ as the \emph{symmetric $q$-shift}. A natural question to ask is when the symmetric $q$-shift satisfies similar properties to the properties of the greedy $q$-shift. As we will see along the paper, the behaviours of the subshifts $(\Sigma_q, \si)$ and $(\vb_q, \si)$ are completely different. In this direction, it was shown in \cite{KomKonLi2015} (see also \cite{AllKon2018}) that the entropy function, i.e. $H:(1, M+1] \to [0, 1]$ given by $H(q) = h_{\operatorname{top}}(\vb_q)$, is a devil staircase. In \cite{AlcBakKon2016} a full description of the plateaus of the entropy function, as well as a full description of its bifurcation set, was given. Also, in \cite[Theorem 1]{AlcBakKon2016} the set of $q$'s such that the symmetric $q$-shift is a transitive subshift was characterised using the symbolic properties of the quasi-greedy expansion of $1$ in base $q$. It is also worth mentioning that in \cite{AlcBakKon2016} it was shown that the sets 
$$\Tr = \set{q \in [q_G(M), M+1] : (\vb_q, \sigma) \text{ is topologically transitive}}$$ 
and 
$$\NTr = \set{q \in [q_G(M), M+1] : (\vb_q, \si) \text{ is not topologically transitive}}$$ 
both have positive Lebesgue measure. Furthermore, for every $q \in (q_G(M), q_T(M))$, the subshift $(\vb_q, \si)$ is not transitive.
We will give a brief explanation of the constants $q_G(M)$ and $q_T(M)$ in Section \ref{sec:prelim}.

The objective of the paper is to solve some questions similar to the ones posed by Blanchard in \cite{Bla1989} in the context of symmetric $q$-shifts and to develop a similar result to Theorem \ref{teo:Schmeling} ---see \cite[p. 693]{Sch1997}. For this purpose, inspired by \cite{Bla1989} we introduce the following classes of subshifts: 

\begin{definition}\label{def:classesofsubshifts}
Let 
$$\vl = \vl(M) = \set{q \in (1, M+1]: \overline{\al(q)} \preccurlyeq \sigma^n(\al(q)) \preccurlyeq \al(q) \textrm{ for every } n \geq 0}.$$ 
We define the following classes of symmetric $q$-shifts:
\begin{align*}
\Ci &= \Ci(M) = \set{q \in \vl: (\vb_q, \si) \textrm{ is a subshift of finite type} }\\
\Cii &= \Cii(M) = \set{q \in \vl: (\vb_q, \si) \textrm{ is a strictly sofic subshift} }\\
\Ciii &= \Ciii(M) = \set{q \in \vl: (\vb_q, \si) \textrm{ is a subshift with the specification property} }\\ 
\Civ &= \Civ(M) = \set{q \in \vl: (\vb_q, \si) \textrm{ is a synchronised subshift} }\\
\Cv &= \Cv(M) = \vl \setminus \Civ.
\end{align*}
\end{definition}

From \cite[Theorem 1.7]{DeVKom2008} it is not difficult to see that for Lebesgue-almost-every $q \in (1,M+1]$, the subshift $(\vb_q, \sigma)$ is a subshift of finite type. However, we can ask about the size of the classes $\Ci, \Cii, \Ciii, \Civ \textrm{ and } \Cv$ and their topological structure as subsets of $\vl$. Also, we want to understand what the symbolic properties of $\al(q)$ are when $q$ belongs to one of the considered classes.  

Using the results obtained by De Vries and Komornik in \cite[Theorem 1.7, 1.8]{DeVKom2008} (see also \cite[Theorem 1.3]{LiSahSam2016}) it is immediate that the class $\Ci$ is countable and dense in $\vl$. Moreover, $q \in \Ci$ if and only if $\al(q)$ is periodic. As a consequence of \cite[Proposition 2.14]{KalSte2012}, Kalle and Steiner characterised the class $\Cii$; namely, $q \in \Cii$ if and only if $\al(q)$ is eventually periodic. It is not difficult to check that $\Cii$ is a countable set and that 
$$\Cii \subset \ul = \ul(M) = \set{q \in (1, M+1] : \overline{\al(q)} \prec \si^n(\al(q)) \prec \al(q) \text{ for every } n \geq 0} \cup \set{M+1}.$$ 

The main result of our work is the following. 
\begin{main}\label{th:symshift1}
\emph{Let $M \in \mathbb{N}$, $q \in [q_{G}(M), M+1]$ and consider $(\vb_{q}, \sigma)$ the symmetric $q$-shift. Then:
\begin{enumerate}[$i)$]
\item $(\vb_{q}, \si)$ has the specification property if and only if $\al(q)$ is a strongly irreducible sequence (see Definition \ref{def:strongweaksequences}) and there exists $K \in \N$ such that $d(\si^k(\al(q)), \overline{\al(q)}) \geq 1/2^K$ for every $k \in \N$.
\item $(\vb_q, \si)$ is synchronised if and only if $\al(q)$ is an irreducible sequence and $\al(q), \, \overline{\al(q)}$ are not dense in $\vb_q$. Moreover $\dim_{\operatorname{H}}(\Civ) = 1$. 
\item $\dim_{\operatorname{H}}(\Cv) = 1.$
\end{enumerate}
}
\end{main}

The structure of the paper is the following. In Section \ref{sec:prelim} we recall the relevant concepts of symbolic dynamics and unique $q$-expansions needed to develop our investigation. In Section \ref{sec:approx} we introduce \emph{the natural approximation from below} of a symmetric subshift $(\vb_q, \si)$. Using this approximation we prove that every transitive symmetric subshift is mixing and coded. In Section \ref{sec:topologicaldynamics} we will characterise the elements of the class $\Ciii$. In Section \ref{sec:synchr} we will study the classes $\Civ$ and $\Cv$. Finally, in Section \ref{sec:hd} we will calculate the Hausdorff dimension of the classes $\Ciii$, $\Civ$ and $\Cv$.  

\section{Preliminaries}
\label{sec:prelim}

\noindent In this section, we recall some basic tools and definitions used in our study. We will adopt most of the notation used in \cite{AlcBakKon2016} and \cite{KalKonLiLu2017} and in Subsection \ref{subsec:betaexpansions} we summarise the  results in those papers relevant for this work.

We refer the reader to \cite{LinMar1995} for a thorough exposition of symbolic dynamics. A standard reference for the theory of Hausdorff dimension is \cite{Fal1990}. Finally, detailed works on unique expansions in non-integer bases are \cite{DarKat1995, DeVKom2008, Kom2011, KomLor2007}. 

\subsection{Symbolic Dynamics}
\label{subsec:symbolic}

We recall some basic notions of symbolic dynamics firstly. Fix $M \in \N$. We call the set $\set{0, \ldots, M}$ an \emph{alphabet} and its elements are called \emph{symbols}. A \emph{word} $\om = w_1 \ldots w_n$ is a finite string of symbols. We denote the \emph{length of a word $\om$} by $|\om|$. Given two words $\om = w_1 \ldots w_n$ and $\nu = v_1 \ldots v_m$ their \emph{concatenation} is the word $\om \nu = w_1 \ldots w_n v_1 \ldots v_m$. Also, $\om^k$ is the word obtained by concatenating $\om$ with itself $k$ times and $\om^\f$ is the infinite concatenation of $\om$ with itself. We denote the \emph{empty word }by $\epsilon$.

We consider 
\[
\Sig = \mathop{\prod}\limits_{n=1}^{\f} \set{0, 1, \ldots, M},
\] 
i.e. $\Sig$ is the set of infinite one-sided sequences whose symbols belong to $\set{0, \ldots, M}$. It is a well known fact that $\Sig$ is a compact space with the product topology. Also, the product topology on $\Sig$ is equivalent to the topology induced by the distance given by 

\begin{equation}\label{eq:metric d}
d(\x,\y) = \left\{
\begin{array}{clrr}
2^{-j} & \hbox{\rm{ if }}  \x \neq \y, & \hbox{\rm{where }} j = \min\{i : x_i \neq y_i \};\\
0 & \hbox{\rm{ otherwise.}}&\\
\end{array}
\right.
\end{equation}

The \emph{one-sided shift map} $\si: \Sig \to \Sig$ is given by $\sigma(\x) = \sigma((x_{i})) = (x_{i+1})$. We call the pair $(\Sig, \si)$ a \emph{full one-sided shift}. We call a sequence $\x \in \Sig$ \emph{periodic} if there exist $m \in \mathbb{N}$ such that $\sigma^m(\x) = \x$. The smallest $m$ satisfying this property is called \emph{the period of $\x$} and the word $x_1 \ldots x_m$ is called \emph{the periodic block of $\x$}. A sequence $\x \in \Sig$ is said to be \emph{eventually periodic} if there exist $m \in \mathbb{N}$ such that $\si^m(\x)$ is a periodic sequence and $\sigma^n(\x)$ is not a periodic sequence for every $0 \leq n \leq m-1$.

A word $\omega = w_1 \ldots w_n$ is called a \emph{factor of a sequence $\x \in \Sig$} if there exists $k \in \mathbb{N}$ such that $x_{k+1} \ldots x_{k+n} = w_1 \ldots w_n$. If $k = 1$ we say that $\omega$ is a \emph{prefix of $\x$}. Note that the notions of factor and prefix can be defined on finite words in a similar fashion. Then, if a word $\omega$ of length $n$ is a factor of a word $\nu$ and $w_1 \ldots w_n = v_{m-n+1} \ldots v_m$ where $|\nu| = m$ we call $\omega$ a \emph{suffix of $\nu$ of length $n$}.   

Given a sequence $\x \in \Sig$ we define its \emph{reflection} as $\overline{\x} = (\overline{x_i}) = (M-x_i).$ For a word $\om$ we define the \emph{reflection of $\om$} similarly. In this case $\overline{\om} = \overline{w_1 \ldots w_n} = \overline{w_1} \ldots \overline{w_n}.$ If $\om$ is a word of length $n$ such that $w_n < M$ we write $$\om^+ = w_1 \ldots w_n^+ = w_1 \ldots w_{n-1} w_n+1$$ and if $\om$ is a word of length $n$ with $w_n > 0$ we put $$\om^- = w_1 \ldots w_n^- = w_1 \ldots w_{n-1} w_n-1.$$ 

Throughout this work, we use the \emph{lexicographic order}. Given $\x, \y \in \Sigma_M$ we say that $\x \prec \y$ if there exist $n \in \N$ such that $x_i = y_i$ for every $i \leq n -1$, and $x_n < y_n$. Also, we write $\x \lle \y$ if $\x = \y$ or $\x \prec \y$. We can define the lexicographic order between words of the same length in a similar way.

A \emph{subshift} $(X, \si)$ is a pair where $X$ is a non-empty, closed and forward-$\si$-invariant subset of $\Sig$ and $\si$ is understood to be $\si_{\mid_X}$. From \cite[Theorem 6.1.21]{LinMar1995} we have that for every closed, non-empty and forward-$\si$-invariant subset $X$ of $\Sig$ there is a set of finite words $\F$ with symbols in $\set{0, \ldots , M}$ such that $X = X_{\F}$ where 
$$X_{\F} = \set{\x \in \Sigma_M : \om \textrm{ is not a  factor of } \x \textrm{ for any } \omega \in \F}.$$ 
If $\F$ can be chosen to be finite we say that $(X, \si)$ is a \emph{subshift of finite type}. We say that a subshift $(X, \si)$ is a \emph{sofic subshift} if it is a factor of a subshift of finite type, i.e. there exists a subshift of finite type $(X', \sigma)$ (not necessarily in the same alphabet) and a semi-conjugacy $h: X' \to X$, i.e. a continuous and surjective map such that $h \circ \si_{\mid_{X'}} = \si_{\mid_{X}} \circ h$. Alternatively, a subshift $(X,\si)$ is sofic if there is a labelled graph $\pazocal{G} = (\mathcal{G}, \mathcal{E})$ which represents $(X, \si)$.

For a subshift $(X, \si)$ and $n \in \mathbb{N}$, \emph{the set of admissible words of length $n$} is given by 
$$B_n(X) = \set{\omega : \omega \textrm{ is a factor for some } \x \in X \textrm{ and } |\omega| = n}.$$ 
For $n = 0$, $B_n(X) = \set{\epsilon}$. We define \emph{the language of $X$} by 
$$\La(X) = \mathop{\bigcup}\limits_{n=0}^{\f}B_n(X).$$ 
Given $M \in \mathbb{N}$ and a subshift $(X, \si)$ of $\Sig$ we define \emph{the topological entropy of $(X, \si)$} by $$h_{\operatorname{top}}(X) = \mathop{\lim}\limits_{n \to \f} (1/n) \log (\# B_n(X))$$ where $\log = \log_{M+1}$ and $\#$ denotes the cardinality of a set.  

Given a subshift $(X,\si)$ and $\om \in \La(X)$ we define the \emph{follower set of $\om$} to be 
$$F_X(\om) = \set{\nu \in \La(X): \om\nu \in \La(X)};$$ 
and the \emph{prefix set of $\om$} to be 
$$P_X(\om) = \set{\upsilon \in \La(X) : \upsilon \om \in \La(X)}.$$ 
Given $m \in \N$ and a word $\om \in \La(X)$ we denote by $F^m_X(\om) = \set{\upsilon \in F_X(\om): |\upsilon| = m}$. 

\begin{definition}\label{def:topologicalproperties}
We say a subshift $(X, \si)$:
\vspace{0.2em} 
\begin{enumerate}[$i)$]
\setlength\itemsep{0.7em}
\item is \emph{topologically transitive}, if for every ordered pair of words $\upsilon, \nu \in \La(X)$ there is a word $\omega \in \La(X)$ such that $\upsilon\omega\nu \in \La(X)$. This is equivalent to having a point $\x \in X$ such that $\set{\si^n(\x)}_{n = 0}^\f$ is a dense subset of $X$;

\item is \emph{topologically mixing}, if for every ordered pair of words $\upsilon, \nu \in \La(X)$ there exists $N = N(\upsilon, \nu) \in \N$ such that for every $n \geq N$ there is a word $\omega \in B_n(X)$ such that $\upsilon \omega \nu \in \La(X)$;

\item has \emph{the specification property}, or simply \emph{has specification}, if there exist $S \in \mathbb{N}$ such that for any $\upsilon, \nu \in \La(X)$ there exist $\om \in B_S(X)$ such that $\upsilon\om\nu \in \La(X)$, i.e. every two words $\upsilon$ and $\nu$ can be connected by a word $\om$ of length $S$;

\item has the \emph{almost specification property}, or simply \emph{A-specification}, if there exist $S \in \mathbb{N}$ such that for any $\upsilon,\nu \in \La(X)$ there exists $\om \in \La(X)$ such that $\upsilon\om\nu \in \La(X)$ and $|\om| \leq S$;

\item is a \textit{coded system}, if 
$$X = \overline{\mathop \bigcup \limits_{n=1}^{\infty} X_n},$$
where $\overline{\mathop \bigcup \limits_{n=1}^{\infty} X_n}$ denotes the topological closure of $\mathop \bigcup \limits_{n=1}^{\infty} X_n$, each $(X_n, \sigma)$ is a transitive subshift of finite type, and $X_n \subset X_{n+1}$ for every $n \in \N$ - see\cite[Theorem 2.1]{FieFie2001} and \cite{Kri2000}.
\end{enumerate}
\end{definition}

It is well known that transitive subshifts of finite type and transitive sofic subshifts are coded. Also, observe that every subshift $(X, \si)$ with specification is topologically transitive, and in fact mixing. It is easy to show that the almost specification property coincides with the specification property if $(X, \si)$ is topologically mixing. Finally, we would like to remark that all coded systems are topologically transitive \cite{Bla1989}.

We now state the specification property in a more convenient way for our purposes. Given a subshift $(X, \si)$ we define 
\begin{align*}
s_n = s_n(X) = \inf \{k \in \N : &\hbox{\rm{ for every }} \upsilon, \nu \in B_n(X) \hbox{\rm{ there exists }} \omega \in B_k(X)\\ 
&\hbox{\rm{ such that }} \upsilon \omega \nu \in \La(X) \}.
\end{align*}

Then $(X, \si)$ has specification if and only if $\mathop{\lim}\limits_{n \to \infty} s_n < \f$. If such limit exists we call it \emph{the specification number of $(X, \si)$} and we denote it by $s_X$.

Given a transitive subshift $(X,\si)$ a word $\om \in \La(X)$ is said to be \emph{intrinsically synchronising} if whenever $\upsilon\omega$ and $\omega\nu \in \La(X)$ we have $\upsilon\omega\nu \in \La(X).$ A transitive subshift $(X, \si)$ is \emph{synchronised} if there exists an intrinsically synchronising word $\om \in \La(X)$. 

Finally, we list the following list of implications stated in \cite{Boy2000} for topologically mixing subshifts:

\begin{equation}\label{eq:sequenceofshifts}
\text{Finite Type } \Rightarrow \text{ Sofic }\Rightarrow \text{ Specification } \Leftrightarrow \text{ A-specification }\Rightarrow \text{ Synchronised }\Rightarrow \text{ Coded.}
\end{equation}

\subsection{Hausdorff Dimension}
\label{subsec:hd}

Let $A$ be a subset of a metric space $X$. Recall that a collection of subsets of $X$, $\pazocal{U} = \set{U_{\lambda}}_{\lambda \in\Lambda}$
is an \emph{open cover of $A$} if each $U_{\lambda} \in \pazocal{U}$ is an open set and $$A \subset \mathop{\bigcup}\limits_{\lambda \in \Lambda}U_{\lambda}.$$ 
If $\pazocal{U} = \set{U_{i}}_{i=1}^\f$ is a countable open cover of $A$ and $\operatorname{diam}(U_i) \leq \delta$ for a given $\delta > 0$ we say that $\pazocal{U}$ is a \emph{$\delta$-cover of $A$}. 

Fix $s\geq 0$. For $\delta>0$ we set 
\[
\pazocal{H}^{s}_{\delta}(A):=
\inf\left\{\sum_{i=1}^{\infty} \text{diam}(U_{i})^{s}:
\set{U_{i}}_{i=1}^\f 
\text{is a }\delta-\text{cover of } A \right\}.
\]
The \emph{$s$-dimensional Hausdorff measure of $A$} is defined to be 
$$\pazocal{H}^{s}(A):=\mathop{\lim}\limits_{\delta\to 0} \pazocal{H}^{s}_{\delta}(A).$$ 
The \emph{Hausdorff dimension of $A$} is given by
$$\dim_{\operatorname{H}}(A)=\inf\{s:\pazocal{H}^{s}(A)=0\}=\sup \{s:\pazocal{H}^{s}(A)=\infty\}.$$ 
We will only consider $X = \mathbb{R}$ with the usual topology in this work.

Given a subset $A \subset \mathbb{R}$ and $x \in A$ we define, following \cite{Urb1987}, \emph{the local Hausdorff dimension of $A$ at $x$} to be
$$\dim_{\operatorname{H}}^{\operatorname{loc}}(A, x) = \mathop{\lim}\limits_{\delta \to 0} \dim_{\operatorname{H}}((x - \delta, x + \delta) \cap A).$$ 
The following standard result is useful to compute the Hausdorff dimension of a subset of $\mathbb{R}$. 

\begin{lemma}\label{th:hdderong}
Let $X$ be a metric space and $A \subset X$ be a compact subset. Then, 
$$\dim_{\operatorname{H}}(A) = \mathop{\sup}\limits_{x \in A} \set{\dim_{\operatorname{H}}^{\operatorname{loc}}(A, x)}.$$ 
Moreover, if $A\subset X$ satisfies that its topological closure $\overline{A}$ is compact, then 
$$\dim_{\operatorname{H}}(A) = \mathop{\sup}\limits_{x \in A} \set{\dim_{\operatorname{H}}^{loc}(A, x)}$$ 
if the map $x \longmapsto \dim_{\operatorname{H}}^{\operatorname{loc}}(A,x)$ is continuous on $\overline{A}$. 
\end{lemma}

\subsection{Expansions in non-integer bases}
\label{subsec:betaexpansions}

Let us bring up to mind the properties of $q$-expansions used in our study. 

Fix $M \in \N$ and let $q \in (1, M+1]$. The \textit{greedy $q$-expansion} of $\x \in I_{M,q} = \left[0, M/q -1\right]$ is the lexicographically largest $q$-expansion of $x$, and the \textit{quasi-greedy $q$-expansion of $x \in I_{M,q} \setminus \{0\}$} is the lexicographically largest $q$-expansion of $x$ with infinitely many non-zero elements. We denote by $\beta(q) = (\beta_i(q))$ the greedy $q$-expansion of $1$, and the quasi-greedy $q$-expansion of $1$ is denoted by $\al(q) = (\al_i(q))$. It is not hard to check that if $\beta(q)$ is not a finite sequence, then $\beta(q)= \al(q)$ and if $\beta(q)$ is a finite sequence then 
$$\al(q) = (\beta_1(q) \ldots \beta_{k-1}(q) \beta_k(q)^-)^\f$$ 
where $k$ satisfies that $\beta_j(q) = 0$ for every $j > k$. We define 
$$\vl = \vl(M)=\set{q\in(1, M+1]: \overline{\al(q)}\lle\si^n(\al(q))\lle \al(q)\textrm{ for all } n\ge 0}$$ and $$\vb = \vb(M) = \set{\al \in \Sig : \overline{\al}\lle\si^n(\al)\lle \al \textrm{ for all } n\ge 0}.$$ 
The following result is essentially due to Parry \cite{Par1960}; see also \cite[Theorem 2.5]{BaiKom2007}, \cite[Proposition 2.3]{DeVKomLor2016}.

\begin{lemma}\label{lem:quasigreedyexpansion}
The map $\Phi:\vl \to \vb$ given by $\Phi(q) = \al(q)$ is strictly increasing and bijective. Moreover, $\Phi$ is continuous from the left and $\Phi^{-1}$ is strictly increasing, bijective and continuous.
\end{lemma}

Let us remind the reader that for every $q\in(1, M+1]$ the \emph{univoque set on base $q$}  is given by $$\u_q:=\{x\in I_{M,q}: x \textrm{ has a unique }q\textrm{-expansion}\},$$ and $\ub_q:=\pi_{q}^{-1}(\u_{q})$ is the corresponding set of $q$-expansions. It was shown in \cite{DeVKom2008} that every sequence $\x \in \ub_q$ satisfies the following lexicographic inequalities:
$$\left\{
\begin{array}{lll}
\si^n(\x)\prec \al(q) &\textrm{whenever}& x_n<M,\\
\si^n(\x) \succ \overline{\al(q)} &\textrm{whenever}& x_n>0.
\end{array}
\right.$$ 

Also, recall that \emph{the set of univoque bases} is given by 
$$\ul = \ul(M):=\{q\in(1,M+1]: 1 \textrm{ has a unique } q\textrm{-expansion}\}.$$ 
The set $\ul$ has Lebesgue measure zero and full Hausdorff dimension---see \cite{ErdJoo1992, DarKat1995, KomKonLi2015}. More\-over, the topological closure of $\ul$, $\overline{\ul}$, is a \emph{Cantor set} \cite[Theorem 1.2]{DeVKomLor2016}. The sets $\ul$ and $\overline{\ul}$ were characterised symbolically in \cite[Theorem 2.5, Theorem 3.9]{DeVKomLor2016} as follows:

\begin{lemma}\label{lem:uandclosureu}
Let $M \in \mathbb{N}$. Then:
\begin{enumerate}[$i)$]
\item $\ul=\set{q\in(1,M+1): \overline{\al(q)}\prec\si^n(\al(q))\prec \al(q)\textrm{ for all }n\ge 1}\cup\set{M+1};$
\item $\overline{\ul}=\set{q\in(1,M+1]: \overline{\al(q)}\prec \si^n(\al(q))\lle \al(q)\textrm{ for all }n\ge 0}.$
\end{enumerate}
\end{lemma}

Clearly $\ul\subsetneq\overline{\ul}\subsetneq\vl$. The topological and symbolic properties of $\ul$,  $\overline{\ul}$ and $\vl$ are summarised in the following theorem---see \cite{DeVKom2008, DeVKomLor2016}. 

\begin{theorem}\label{thm:uclosureuv}
Let $M \in \mathbb{N}$ then:
\begin{enumerate}[$i)$]
\item $\overline{\ul}\setminus\ul$ and $\vl\setminus\overline{\ul}$ are both countable; 
\item $\overline{\ul}\setminus\ul$ is dense in $\overline{\ul}$. Moreover, if $q \in \overline{\ul}\setminus\ul$ then $\al(q)$ is periodic;
\item $\vl\setminus\overline{\ul}$ is discrete and dense in $\vl$. Moreover, if $q \in \vl\setminus \overline{\ul}$ then $\al(q)$ is periodic.
\item If $q \in \vl\setminus\ul$ then $\al(q)$ is periodic. 
\end{enumerate}  
\end{theorem}
 
It follows from \cite[Lemma 3.5]{DeVKomLor2016} that the smallest element of $\vl$ is the \emph{generalised golden ratio}, denoted by $q_G = q_G(M)$, and defined as 
\begin{equation}\label{eq:21}
 q_G(M) =\left\{\begin{array}{lll}
k+1 &\textrm{ if } & M=2k,\\
 {k+1+\sqrt{k^2+6k+5}}/2&\textrm{ if }& M=2k+1.
\end{array}\right.
\end{equation}
Furthermore $\al(q_G)=k^\f$ if $M=2k$ and $\al(q_G)=((k+1)k)^\f$ if $M=2k+1$. Thus, $q_G \in \vl \setminus \overline{\ul}$. Also, it was shown in \cite{KomLor2002} that the smallest element of $\ul$ is the so-called \emph{Komornik–Loreti constant}, denoted by $q_{KL} = q_{KL}(M)$, which is defined using the classical \emph{Thue-Morse sequence} $(\tau_i)_{i=0}^\f$; this sequence is defined as follows: $\tau_0=0$, and if $\tau_i$ has already been defined for some $i\ge 0$, then $\tau_{2i}=\tau_i$ and $\tau_{2i+1}=1-\tau_i$. The Komornik-Loreti constant is defined explicitly by 
\begin{equation}\label{eq:23}
\al(q_{KL})=(\lambda_i)_{i=1}^\f:=\left\{
\begin{array}{lll}
(k+\tau_i-\tau_{i-1})_{i=1}^\f&\textrm{if}& M=2k,\\
(k+\tau_i)_{i=1}^\f&\textrm{if}& M=2k+1.
\end{array}
\right.
\end{equation}

Notice that the sequence $(\lambda_i)_{i=1}^\f$ in \eqref{eq:23} depends on $M$. Also, from the definition of $(\tau_i)_{i=0}^\f$ it follows that $(\lambda_i)$ satisfies the recursive equations:
\begin{equation}\label{eq:24}
\lambda_1\ldots \lambda_{2^{n+1}}=\lambda_1\ldots \lambda_{2^n}\overline{\lambda_1\ldots \lambda_{2^n}}^+\quad\textrm{for all}\quad n\ge 0.
\end{equation}
\cite{KomLor2002}. Therefore, $\al(q_{KL})$ starts with
\begin{align*}
&(k+1)k(k-1)(k+1)\; (k-1)k(k+1)k\cdots&\textrm{if}&\quad M=2k,\\
&(k+1)(k+1)k(k+1)\; k k (k+1)(k+1)\cdots&\textrm{if}&\quad M=2k+1.
\end{align*}
Using \eqref{eq:21}, \eqref{eq:23} and Lemma \ref{lem:quasigreedyexpansion} we obtain that $q_G < q_{KL} < q_T$. Here $q_T$ refers to the \emph{transitive base} defined in \eqref{eq:transconstant}.

We would like to bring to mind that $(\ub_q, \sigma)$ is not always a subshift \cite[Theorem 1.8]{DeVKom2008}. However we consider the subshift $(\vb_q, \sigma)$, where $q \in \vl$ and
$$\vb_q=\set{\x\in \Sig: \overline{\al(q)}\lle\si^n(\x)\lle\al(q)\textrm{ for all }n\ge 0}.$$ 
Accordingly, we define $\V_q = \pi_q(\vb_q)$. It is easy to check that $\om \in \La(\vb_q)$ if and only if $\om = w_1 \ldots w_n$ satisfies
\begin{equation}\label{eq:lang1}
\overline{\al_1(q) \ldots \al_{n-i}(q)} \lle w_{i+1} \ldots w_n \lle \al_1(q) \ldots \al_{n-i}(q)\text{ for every }0 \leq i \leq n-1.
\end{equation}  
Also, it is clear that if $\om \in \La(\vb_q)$ then $\overline{\om} \in \La(\vb_q)$.

In \cite[Proposition 2.8]{KomKonLi2015} Komornik et. al. showed that
$$h_{\operatorname{top}}(\vb_q)=h_{\operatorname{top}}(\ub_q) \quad\textrm{for all }q\in[q_G, M+1].$$ 
Moreover, in \cite[Theorem 1.3]{KomKonLi2015} the authors showed that 
\begin{equation}\label{eq:hausdoffdim}
\dim_{\operatorname{H}}(\ul_q) = h_{\operatorname{top}}(\vb_q)/\log(q)\quad \textrm{for all } q\in[q_G, M+1].
\end{equation} 

Komornik et. al. also studied the \emph{entropy function} $H:[q_G, M+1] \to [0, 1]$ given by $H(q) = h_{\operatorname{top}}(\vb_q)$ and the \emph{Hausdorff dimension function} $HD:[q_G, M+1] \to [0, 1]$ given by $HD(q) = \dim_{\operatorname{H}}(\vl_q)$. In \cite[Lemma 2.11]{KomKonLi2015} (see also, \cite{AllKon2018} \cite[Theorem 2.6]{KonLi2015}) was shown that $H$ is a devil's staircase and, as a consequence of \ref{eq:hausdoffdim}, it is shown in \cite[Theorem 1.4]{KomKonLi2015} that $HD$ is continuous, has bounded variation and has devil's-staircase-like behaviour. 

In \cite[Theorem 2, Theorem 3]{AlcBakKon2016} the \emph{entropy plateaus}, i.e, the maximal intervals $[p_L, p_R]$ for which
$$H(q)=H(p_L)\quad\textrm{for all }q\in[p_L, p_R],$$ 
as well as the bifurcation set 
$$\B = \set{q\in(1,M+1]:  H(p)\neq H(q) \quad \text{for any }p \neq q}$$ 
and its topological closure 
$$\overline{\B} = \set{q \in (1, M+1]: \text{ for all }\varepsilon > 0 \text{ there exists }p \in (q-\varepsilon, q+\varepsilon) \text{ such that } H(p) \neq H(q)}$$ were characterised. The set 
$$\mathcal{T} = \set{q \in \vl: (\vb_q, \si) \text{ is a transitive subshift }}$$ 
was also characterised (see \cite[Theorem 1]{AlcBakKon2016}).

The following special classes of sequences in $\vb$ were introduced in \cite{AlcBakKon2016} (see also \cite{AlcBar2014}) in order to characterise the entropy plateaus and the sets $\B, \overline{\B}$ and $\mathcal{T}$. 

Let $\zeta_i = \al(q_G(M))_i$. For every $n \in \N$ we define the sequence 
$$\xi(n):=\left\{
\begin{array}
  {lll}
  \zeta_1\ldots\zeta_{2^{n-1}}(\overline{\zeta_1\ldots\zeta_{2^{n-1}}}\,^+)^\f&\textrm{if}& M=2k,\\
  \zeta_1\ldots\zeta_{2^n}(\overline{\zeta_1\ldots\zeta_{2^n}}\,^+)^\f
	&\textrm{if}& M=2k+1.
\end{array}\right.
$$

\begin{definition}\label{def:irreducible}
\leavevmode
\begin{enumerate}[$i)$]
\item A sequence $\al \in \vb$ is said to be \emph{irreducible} if 
$$\al_1\ldots \al_j(\overline{\al_1\ldots \al_j}\,^+)^\f\prec \al$$ 
whenever $(\al_1\ldots \al_j^-)^\f\in {\vb}$.
\item A sequence $\al \in \vb$ is said to be \emph{$*$-irreducible} if there exists $n\in\N$ such that $\xi(n+1)\lle \al \prec \xi(n),$ and 
$$\al_1\ldots \al_j(\overline{\al_1\ldots \al_j}\,^+)^\f\prec \al,$$ 
whenever 
$$(\al_1\ldots \al_j^-)^\f\in\vb\quad\textrm{and}\quad j>
 \left\{\begin{array}{lll}
 2^n&\textrm{if}& M=2k,\\
 2^{n+1}&\textrm{if}& M=2k+1.
 \end{array}\right.
$$
\end{enumerate}  
\end{definition}

We denote by 
$$\I = \set{q \in \vl : \al(q) \text{ is irreducible }}$$ and by  $$\IS= \set{q \in \vl : \al(q) \text{ is }*\text{-irreducible }}.$$ 

We would like to mention that $\I$ and $\IS$ are subsets of $\overline{\ul}$ \cite[Lemma 4.7]{AlcBakKon2016}. It is not difficult to check that $\al(q_G)$ is irreducible, and hence it is the smallest irreducible sequence. Also, the base $q_T = q_T(M)$, called the \emph{transitive base}, was introduced in \cite{AlcBakKon2016}. The base $q_T$ is defined implicitly as

\begin{equation}\label{eq:transconstant} 
\al(q_T)=\left\{
\begin{array}
  {lll}
  (k+1)\,k^\f&\textrm{if}& M=2k,\\
  (k+1)\,((k+1)k)^\f&\textrm{if}&M=2k+1.
\end{array}
\right.
\end{equation}

Notice that $\al(q_T)\in\vb$ and therefore $q_T\in\vl$. Moreover $q_T \in \ul$ and $q_T > q_G$. The following Theorem summarises \cite[Theorem 1, Theorem 2, Theorem 3]{AlcBakKon2016}.

\begin{theorem}\label{th:mainrsd}
\leavevmode
\begin{enumerate}[$i)$]
\item Let $q\in \vl$. Then $(\vb_q, \si)$ is a transitive subshift if and only if  $\al(q)$ is  irreducible, or $q=q_T$;
\item The interval $[q_L, q_R]\subseteq(q_{KL}, M+1]$ is an entropy plateau of $H$ if and only if $q_L \in \I\cup \IS$ and $\al(q_L)$ is periodic, and 
$$\al(q_R) = \al_1(q_L) \ldots \al_m(q_L)^{+}(\overline{\al_1(q_L) \ldots \al_m(q_L)})^\f;$$  
\item The topological closure of the bifurcation set $\B$ is 
$$\overline{\B}= \overline{\I \cup \IS} \subset \overline{\ul}.$$ 
Moreover $\overline{\B}$ is a Cantor set and $\dim_{\operatorname{H}}(\overline{\B})=1$.
\end{enumerate}
\end{theorem}
The interval $[q_L, q_R]$ described in $ii)$ is known as the \emph{irreducible interval generated by $q_L$} whenever $q_L \in \I$. We will denote an irreducible interval generated by $q \in \I$ with $\al(q)$ periodic by $I(q)$, whenever necessary.

In \cite[Remark 1.2]{KalKonLiLu2017} Kalle et. al. mentioned that $\overline{\B}\setminus \B$ is a countable set. Moreover, it is clear that if $q \in \overline{\B}\setminus \B$ then $q$ is an end point of an entropy plateau. Finally in \cite[Theorem 2]{KalKonLiLu2017} it is shown that
$$\dim_{\operatorname{H}}^{\operatorname{loc}}(\overline{\B}, q) = \dim_{\operatorname{H}}(\u_q) = \dim_{\operatorname{H}}(\V_q).$$ 
By \cite[Theorem 2]{KalKonLiLu2017} and \cite[Theorem 1.4]{KomKonLi2015} the following proposition holds.

\begin{proposition}\label{pr:continuityinB}
The functions $H$ and $HD$ are continuous in $\overline{\B}$. Moreover the map 
$$q \to \dim_{\operatorname{H}}^{\operatorname{loc}}(\overline{\B}, q)$$
is continuous in $\overline{\B}$. 
\end{proposition}

\section{Approximation properties of symmetric $q$-shifts}
\label{sec:approx}

\noindent In this section we introduce a notion of approximation of symmetric $q$-shifts that will be useful for the rest of the paper. Using this approximation we show that every transitive symmetric $q$-shift is coded and mixing.

\begin{definition}\label{def:approximations}
Given $q \in \overline{\ul}$ we define \emph{natural approximation of $q$ from below} as the sequence $\left\{q^-_m\right\}_{m=1}^\f \subset \vl$ given by $q^-_1$ defined implicitly by $\al(q^-_1) = (\al_1(q) \ldots \al_{n_1}(q)^-)^\f$ where 
$$n_1 = \min \set{n \in \N : (\al_1(q)\ldots \al_{n}(q)^-)^\f \in \vb};$$ 
and if $q^-_{m-1}$ is already defined then $q^-_m$ is defined implicitly by
\begin{equation}\label{eq:approxbelow}
\al(q^-_m) = (\al_1(q) \ldots \al_{n_m}(q)^-)^\f 
\end{equation} 
where 
$$n_m = \min\set{n \in \N : n > n_{m-1} \text{ and } (\al_1(q) \ldots \al_{n}(q)^-)^\f \in \vb}.$$
\end{definition}

We make the following observation on Definition \ref{def:approximations}. 

\begin{remark}\label{rem:onvl}
Note that Theorem \ref{thm:uclosureuv} $iii)$ implies that for every $q \in (\vl \setminus \overline{\ul}) \setminus \set{q_G}$, $\al(q)$ is a periodic sequence. Set $m$ to be the period of $\al(q)$. Then, since $q \in \vl \setminus \overline{\ul}$ there exists $m' < m$ such that $\si^{m'}(\al(q)) = \overline{\al(q)}$. We claim that, for every $j > m'$ the sequence $(\al_1(q) \ldots \al_j(q)^-)^\f \notin \vb$. Suppose on the contrary that there exists $j > m'$ such that $(\al_1(q) \ldots \al_j(q)^-)^\f \in \vb$. Then, 
\[
\overline{(\al_1(q) \ldots \al_j(q)^-)^\f} \preccurlyeq \si^{m'}((\al_1(q) \ldots \al_j(q)^-)^\f) \preccurlyeq (\al_1(q) \ldots \al_j(q)^-)^\f,\] 
i.e 
\[
\overline{(\al_1(q) \ldots \al_j(q)^-)^\f} \preccurlyeq (\al_{m'+1}(q) \ldots \al_j(q)^-\al_1(q) \ldots \al_{m'}(q))^\f \preccurlyeq (\al_1(q) \ldots \al_j(q)^-)^\f.
\] 
Note that 
\[
\al_{m'+1}(q) \ldots \al_j(q)^- \prec \overline{\al_1(q) \ldots \al_{j-m'}(q)}. 
\]
Then 
\[
\overline{\al_{m'+1}(q) \ldots \al_j(q)^-} \succ \al_1(q) \ldots \al_{j-m'}(q),
\] 
which is a contradiction. 

We have shown that if $q \in (\vl \setminus \overline{\ul}) \setminus \set{q_G}$ there is $N \in \N$ such that for any $m \geq N$ there is no $n_m > n_N$ such that $$(\al_1(q) \ldots \al_{n_m}(q)^-)^\f \in \vb.$$ Thus, for sake of completeness, if $q \in (\vl \setminus \overline{\ul}) \setminus \set{q_G}$ we set the natural approximation from below of $q$ to be the finite sequence $\set{q^-_1, \ldots q^-_N, q^-_{N+1}}$ where $q^-_{N+1} = q$. Also, since $q_G$ is the smallest element of $\vl$ we set the natural approximation from below of $q_G$ to be $\set{q_G}$.    
\end{remark}

In the following propositions we show that indeed, the natural approximation from below approximates a given $q \in \overline{\ul}$. 

\begin{proposition}\label{pr:aproxconvergencebelow} 
Let $q \in \overline{\ul}$. Then, the natural approximation from below of $q$, $\left\{q^-_m\right\}_{m=1}^\f$, satisfies: 
\begin{enumerate}[$i)$]
\item For every $m \in \mathbb{N}$, $q^-_m < q^-_{m+1}$ and $q^-_m < q$. 
\item $q^-_m \mathop{\nearrow}\limits_{m \to \f} q$.
\end{enumerate}
\end{proposition}
\begin{proof}
Let us show $i)$. From Lemma \ref{lem:quasigreedyexpansion} it suffices to show that 
$$\al(q^-_m) \lle \al(q^-_{m+1}) \text{ and } \al(q^-_m) \lle \al(q)$$ 
for every $m \in \N$. From Definition \ref{def:approximations}, we have for every natural number $m$,
$$\al(q^-_m) = (\al_1(q) \ldots  \al_{n_m}(q)^-)^\f \quad \text{and} \quad \al(q^-_{m+1}) = (\al_1(q) \ldots \al_{n_m}(q) \ldots \al_{n_{m+1}}(q)^-)^\f.$$ 
Note that 
$$\al_{n_m}(q^-_m) =  \al_{n_m}(q)^- < \al_{n_m}(q) = \al_{n_m}(q^-_{m+1}),$$ 
so $\al(q^-_m) \prec \al(q^-_{m+1})$ and $q^-_m \leq q^-_{m+1}$. The proof for $q^-_m \leq q$ follows from the same argument.

To show $ii)$, let $q \in \overline{\ul}$ be fixed. Note that that $d(\al(q), \al(q^-_m)) = 1/2^{n_m}$. Then, for a given $\varepsilon > 0$ there is $N \in \N$ such that if $m \geq N$, then 
$$d(\al(q), \al(q^-_m)) = 1/2^{n_m} < 1/2^N < \varepsilon,$$ 
that is $\al(q^-_m) \longrightarrow \al(q)$ as $m \to \infty$. Then, by Lemma \ref{lem:quasigreedyexpansion} we have that $q^-_m \mathop{\nearrow}\limits_{m \to \f} q$.     
\end{proof}

We introduced the natural approximation from below of elements of $\overline{\ul}$ in order to approximate the subshifts $(\vb_q, \si)$ in the following way: we say that $(X, \si)$ is \emph{approximated from below} if there exists a sequence of subshifts of finite type $\set{(X_m, \si)}_{m=1}^\f$ such that $X_m \subset X_{m+1}$ for every $m \in \N$ and 
$$X = \overline{\mathop{\bigcup}\limits_{m=1}^\f X_n}.$$ 

\begin{lemma}\label{lem:approximations}
If $q \in \vl$ then the subshift $(\vb_q, \si)$ is approximated from below by the sequence of subshifts $(\vb_{q^-_m}, \si)$ where $\set{q^-_m}_{m=1}^\f$ is the natural approximation from below of $q$.
\end{lemma}
\begin{proof}
Firstly, let $q \in \vl \setminus \overline{\ul}$. If $q = q_G$ then, using Remark \ref{rem:onvl}, we have that the natural approximation of below of $q_G$ is $\set{q_G}$. Then, the conclusion of the Lemma follows easily. On the other hand, if $q \neq q_G$ from Remark \ref{rem:onvl} we have that there is $N \in \N$ such that the natural approximation from below of $q$, $\set{q^-_m}_{m=1}^\f$ satisfies that $q^-_m = q^-_{N+1} = q$ for all $m \geq N+1$. So, it is straightforward that $(\vb_{q^-_m}, \si)$ approximates from below the subshift $(\vb_q, \si)$.   

Now, consider $q \in \overline{\ul}$ and let $\set{q^-_m}_{m=1}^\f$ be the natural approximation from below of $q$. From Proposition \ref{pr:aproxconvergencebelow} we have that $q^-_m < q^-_{m+1}$ for every $m \in \N$. Since $\set{q^-_m}_{m=1}^\f \subset \vl$, then $\vb_{q^-_m} \subset \vb_{q^-_{m+1}}$. Since $q^-_m < q$ then $\vb_{q^-_m} \subset \vb_{q}$ for every $m \in \N$. This implies that $\mathop{\bigcup}\limits_{m=1}^\f \vb_{q^-_m} \subset \vb_{q}$. Also, recall that $\vb_q$ is a closed subset of $\Sig$, then $\overline{\mathop{\bigcup}\limits_{m=1}^\f \vb_{q^-_m}} \subset \vb_{q}$. On the other hand, consider $\om \in\La (\vb_q)$. Since $q^-_m \nearrow q$ as $m \longrightarrow \f$, there exists $n \in \mathbb{N}$ such that $\om\in \La(\vb_{q^-_m})$. This implies that there exists $\x \in \vb_{q^-_m}$ such that $\om$ is a factor of $\x$. Then $\mathop{\bigcup}\limits_{k=1}^\f\vb_{q^-_m}$ is dense in $\vb_q$ with respect to the metric $d$ defined in \eqref{eq:metric d}. Thus, 
$$\overline{\mathop{\bigcup}\limits_{m=1}^\f \vb_{q^-_m}} = \vb_{q}.$$
\end{proof}

From Definition \ref{def:approximations}, if $q \in \overline{\ul}$ we have $\set{q^-_m}_{m=1}^\f \subset \vl$. We wish to point out that the approximation from below constructed in Lemma \ref{lem:approximations} is similar to the approximation $(\WB_{p_L}, \si)$ considered in \cite{KalKonLiLu2017} and \cite{KomKonLi2015}. One of the advantages of our approach is that it is not necessary to compare finite words and sequences introducing a variation of the lexicographic order. Also, the constructed approximations together with Remark \ref{rem:onvl} allow us to always get strict inclusions when $q \in \overline{\ul}$. That is, since $q \in \overline{\ul}$, $\set{q^-_m}_{m=1}^\f \subset \vl$ and $q^-_n < q^-_m$ for every $n < m$ we have that $\vb_{q^-_n} \subsetneq \vb_{q^-_m}$ for every $n < m$.

We now show that given the natural approximation from below of $q$ it is also possible to approximate the language of $\vb_q$ by the languages of the associated subshifts of each of the elements of the natural approximation from below (compare with \cite[Lemma 3.4]{KalKonLiLu2017}).  

\begin{lemma}\label{lem:aproxwordsbelow}
Let $q \in \vl$ and consider the natural approximation from below of $q$, $\left\{q^-_m\right\}_{m=1}^\f$. Then, for every $k \in \N$ there exists $J \in \N$ such that if $m \geq J$, then 
\[
B_k(\vb_q) = B_k(\vb_{q^-_m}) = B_k(\vb_{q^-_J}).
\]
\end{lemma}
\begin{proof}
Firstly, let us assume that $q \in \vl \setminus \overline{\ul}$. Then $\al(q)$ is a periodic sequence. Let $J \in \N$ be the period block of $\al(q)$. Then, from Lemma \ref{lem:approximations} there is $N \in \N$ such that $q^-_m = q^-_{N+1}$ and $\al(q^-_m) = \al(q)$. Thus, for all $k \geq J$ we have $B_k(\vb_q) = B_k(\vb_{q^-_m}) = B_k(\vb_{q^-_J}).$

Suppose that $q \in \overline{\ul}$. From Proposition \ref{pr:aproxconvergencebelow} we have $q^-_m < q^-_{m+1} < q$, for every $m \in \N$ and $q^-_m \mathop{\nearrow}\limits_{m \to \f} q$. This implies 
$$\vb_{q^-_m} \subsetneq \vb_{q^-_{m+1}} \subsetneq \vb_q.$$
Then, for every $k \in \N$ and $m, J \in \N$ with $m \geq J$ we have 
\[
B_k(\vb_{q^-_J}) \subset B_k(\vb_{q^-_m}) \subset B_k(\vb_q).
\]
Fix $k \in \N$ and let $n_m$ be the period of $\al(q^-_m)$. We claim that $J = \min\set{m \in N: k < n_m}$ satisfies 
\[
B_k(\vb_q) = B_k(\vb_{q^-_m}) = B_k(\vb_{q^-_J}).
\] 
To show this, it suffices to show that $B_k(\vb_q) \subset B_k(\vb_{q^-_J})$. Let $\om \in B_k(\vb_q)$. Then, for every $i \in \set{0, \ldots, k-1}$ 
\[
\overline{\al_1(q) \ldots \al_{k-i}(q)} \lle w_{i+1} \ldots w_k \lle \al_1(q) \ldots \al_{k-i}(q).
\] 
Since $n_K > k$ then $\al_1(q) \ldots \al_k(q) = \al_1(q^-_J) \ldots \al_k(q^-_J)$. This gives 
\[
\overline{\al_1(q^-_J) \ldots \al_{k-i}(q^-_J)} \lle w_{i+1} \ldots w_k \lle \al_1(q^-_J) \ldots \al_{k-i}(q^-_J) \text{ for every } i \in \set{0, \ldots, k-1}
\] 
and the proposition follows.
\end{proof}

We show now the following technical, but important results. We will show that if $q \in \I$ then there must exist infinitely many irreducible elements in the natural approximation from below $\set{q^-_m}_{m=1}^\f$. For this endeavour, we prove the following statement firstly.

\begin{lemma}\label{lem:spec1}
Let $q \in \I$. Then, if there is $j \in \N$ such that $q^-_j > q_T$ and 
$$(\al_1(q^-_j) \ldots \al_{n_j}(q^-_j))^\f = (\al_1(q) \ldots \al_{n_j}(q)^-)^\f$$ 
is not irreducible then there exists a unique $k < j$ and a unique $1 < n_k < n_j$ such that $q^-_k \in \I$ and $q^-_j \in I(q^-_k)$.
\end{lemma} 
\begin{proof}
Let $q \in \I$ and $q^-_j$ be such that $(\al_1(q^-_j) \ldots \al_{n_j}(q^-_j))^\f = (\al_1(q) \ldots \al_{n_j}(q)^-)^\f$ with $q^-_j > q_T$ and $q^-_j \notin \I$. Then, from \cite[Lemma 4.9]{AlcBakKon2016} there exists a unique irreducible interval $I$ such that $q^-_j \in I$. Let $q' \in I$ be such that $I = I(q')$. Since $I(q')$ is an irreducible interval then there exists a word $w_1 \ldots w_m$ such that $\al(q') = (w_1 \ldots w_m)^\f$ and $\al(q)$ is an irreducible sequence. From \cite[Lemma 4.1]{AlcBakKon2016} we can assume without loss of generality that $m$ is the period of $\al(q')$. From the uniqueness of $I$ we get that $w_1 \ldots w_m$ is unique. Since $q^-_j \in I$ we have  
$$(w_1 \ldots w_m)^\f \prec  (\al_1(q^-_j) \ldots \al_{n_j}(q^-_j))^\f \prec w_1 \ldots w_m^+(\overline{w_1 \ldots w_m})^\f.$$ Observe that $n_j \neq m$ as if $n_j = m$ then $(w_1 \ldots w_m)^\f$ is not irreducible, which is a contradiction. On the other hand, if $m > n_j$ then 
\begin{align*}
\al_1(q^-_j)\ldots \al_{n_j}(q^-_j)\al_{n_j+1}(q^-_j)\ldots \al_{m}(q^-_j) &= \al_1(q^-_j)\ldots \al_{n_j}(q^-_j)\al_{1}(q^-_j)\ldots \al_{m-n_j}(q_j)\\ 
&= w_1 \ldots w_{n_j}w_{n_j+1}\ldots w_m \\
&= w_1 \ldots w_{n_j}w_{1} \ldots w_{m-{n_j}}. 
\end{align*}
This implies that the period of $(w_1 \ldots w_m)^\f$ is smaller than $m$, which is a contradiction as well. Therefore, $m < n_j$. Since $\I \subset \overline{\ul}$, 
$$w_1 \ldots w_m^+ = \al_1(q^-_j) \ldots \al_{m}(q^-_j) = \al_1(q) \ldots \al_m(q).$$ 
This implies that 
$$w_1 \ldots w_m = \al_1(q^-_j) \ldots \al_{m}(q^-_j)^- = \al_1(q) \ldots \al_m(q)^-.$$ 
Then $q^-_k = q'$ and $m = n_k$ satisfy the desired properties.
\end{proof}

\begin{lemma}\label{lem:infinitelymanyirreducibles}
Let $q \in \I$ and $\set{q^-_m}_{m=1}^\f$ be the natural approximation from below of $q$. Then, there exist infinitely many $m \in \N$ such that $q^-_m \in \I$.
\end{lemma}
\begin{proof}
Let $q \in \I$. Let us assume on the contrary that there is $N_1 \in \N$ such that $q^-_m \notin \I$ for every $m \geq N_1$. 

From Proposition \ref{pr:aproxconvergencebelow}, $q^-_m \mathop{\nearrow} q$ as $m \to \f$ and since $q \in \I$ from \cite[Lemma 4.4]{AlcBakKon2016} we know that $q > q_T$. Then, there exists a minimal $N_2 \in \N$ such that, $q^-_m \in (q_T, M+1)$ for every $m \geq N_2$. Let $N = \max\set{N_1, N_2}$. Then, from Lemma \ref{lem:spec1} there is a unique $k < N$ such that $q^-_k \in \I$ and $q^-_N \in I(q^-_{k})$.

We claim that for every $m > N$, $q^-_m \in I(q^-_{k})$. Suppose this is not true. Then, there is $m > N$ such that $q^-_m \notin I(q^-_{k})$. Also note that Lemma \ref{pr:aproxconvergencebelow} implies that $q^-_N < q^-_m$. Then, Lemma \ref{lem:spec1} implies that there is a unique $k' < m$ such that $q^-_{k'} \in \I$ and $q^-_m \in I(q^-_{k'})$. Clearly $k < k'$. From \cite[Lemma 4.6]{AlcBakKon2016} we know that $I(q^-_k) \cap I(q^-_{k'}) = \emptyset$. This implies that $N_1 < k'$ which is a contradiction. Therefore, $q^-_m \in I(q^-_{k})$ for every $m \geq N$ which gives that that $q \in I(q^-_{k})$, thus $q \notin \I$. This establishes a contradiction. 
\end{proof}

We show now that every transitive symmetric $q$-shift $(\vb_q, \si)$ is a coded system. We want to emphasise that $\I \subset \overline{\B} \subset \overline{\ul}$ ---see \cite[Lemma 4.7, Lemma 6.1]{AlcBakKon2016}.

\begin{proposition}\label{pr:coded}
For every $q \in \I$ the subshift $(\vb_q, \si)$ is coded. 
\end{proposition}
\begin{proof}
Let $q \in \I$. Then, from Theorem \ref{th:mainrsd} $i)$ the subshift $(\vb_q, \si)$ is transitive. Then, if $\al(q)$ is periodic then \cite[Theorem 1.7, 1.8]{DeVKom2008} implies that $(\vb_q, \si)$ is coded. Similarly if $\al(q)$ is eventually periodic then \cite[Proposition 2.14]{KalSte2012} implies that $(\vb_q, \si)$ is coded. So, let $q \in \I$ such that $\al(q)$ is neither periodic nor eventually periodic. Then, from Lemma \ref{lem:infinitelymanyirreducibles} and \cite[Lemma 6.1]{AlcBakKon2016} the natural approximation from below $\set{q^-_m}_{m=1}^\f$contains a subsequence $\set{q^-_{m_j}}_{j=1}^\f$ such that $q^-_{m_j} \in \I$, $\al(q^-_{m_j})$ is periodic for every $j \in \N$ and $q^-_{m_j} \mathop{\nearrow}\limits_{j \to \infty} q$. This implies that $\vb_{q^-_{m_j}} \subsetneq \vb_{q^-_{m_{j+1}}}$. Then, from Lemma \ref{lem:approximations} it follows that the sequence of subshifts $\set{(\vb_{q^-_{m_j}}, \si)}_{j=1}^\f$ also approximates $(\vb_q,\si)$ from below. Moreover, since $q^-_{m_j} \in \I$ for every $j \in \N$ we obtain from \cite[Theorem 1]{AlcBakKon2016} that $(\vb_{q^-_{m_j}}, \si)$ is a transitive subshift for every $j \in \N$. Finally, using \cite[Theorem 1.7, 1.8]{DeVKom2008} and \cite[Theorem 1.3]{LiSahSam2016} imply that $(\vb_{q^-_{m_j}}, \si)$ is a subshift of finite type for every $j \in \N$. Thus, $(\vb_q, \si)$ is coded. 
\end{proof}

To finish this section, we now show that every symmetric and transitive $q$-shift is topologically mixing. For this purpose we want to recall the usual Sharkovski\v{i} order of $\N$:

\begin{align*} 
 & 3 & \vartriangleright & 5 & \vartriangleright & 7  & \vartriangleright & \ldots & \vartriangleright & 2m + 1 &  \vartriangleright & \ldots\\
\vartriangleright & 2 \cdot 3 & \vartriangleright & 2\cdot 5 & \vartriangleright & 2 \cdot 7 & \vartriangleright & \ldots & \vartriangleright & 2\cdot(2m + 1) & \vartriangleright & \ldots\\
\vartriangleright & 4 \cdot 3 & \vartriangleright & 4\cdot 5 & \vartriangleright & 4 \cdot 7 & \vartriangleright & \ldots & \vartriangleright & 4\cdot(2m + 1)& \vartriangleright & \ldots\\
& \vdots &  & \vdots &  &\vdots& & & & \vdots&  & \\
\vartriangleright & 2^n \cdot 3 & \vartriangleright & 2^n\cdot 5 & \vartriangleright &2^n \cdot 7& \vartriangleright& \ldots & \vartriangleright & 2^n\cdot(2m + 1)& \vartriangleright & \ldots\\
& \vdots &  & \vdots &  &\vdots& & & & \vdots&  & \\
\vartriangleright & 2^\f   \ldots &  \vartriangleright &2^n & \vartriangleright &  \ldots &\vartriangleright&8& \vartriangleright& 4 & \vartriangleright & 2& \vartriangleright & 1.\\
\end{align*}

It has been shown in \cite[Theorem 1.3]{AllClaSid2009} and \cite[Theorem 1.1]{GeTan2017} that periodic points of $(\vb_q, \si)$ grow with respect to the Sharkowsk\v{i} order of $\N$, that is, if $\vb_q$ contains a periodic point of period $m$ with respect to $\si$, then $\vb_q$ contain points of period $n$ for every $n \vartriangleleft m$. On the other hand it is known (\cite[Proposition 4.5.10 (4)]{LinMar1995}) that a subshift of finite type $(X,\si)$ is topologically mixing if and only if it is transitive and the greatest common divisor of the periods of its periodic points is $1$; that is, there exists a pair of periodic points $\x$ and $\y \in X$ such that $\gcd(m, n) = 1$, where $m$ and $n$ are the periods of $\x$ and $\y$ respectively. Using \eqref{eq:transconstant} it follows that if $q \in \I$ then $(\vb_q, \si)$ contains a periodic orbit of odd period. As a consequence of these results we obtain the following: 

\begin{proposition}\label{pr:mixingsft}
If $q \in \I$ and $\al(q)$ is periodic then $(\vb_q, \si)$ is a mixing subshift of finite type. 
\end{proposition}

Now we prove that every transitive symmetric $q$-shift is mixing.

\begin{proposition}\label{pr:mixing}
If  $q \in \I$ then $(\vb_q, \si)$ is a mixing subshift. 
\end{proposition}
\begin{proof}
From Proposition \ref{pr:mixingsft} we can assume that $\al(q)$ is not periodic. From Lemma \ref{lem:infinitelymanyirreducibles} we have that there is a subsequence $\set{q^-_{m_j}}_{j=1}^\f$ of the natural approximation from below of $q$, $\set{q^-_m}_{m=1}^\f$ such that $(\vb_{q^-_{m_j}}, \si)$ is a mixing subshift of finite type for every $j \in \N$. Also, Lemma \ref{lem:approximations} implies that $(\vb_q, \si)$ is approximated from below by $\set{(\vb_{q_{m_j}}, \si)}_{j=1}^\f$. Let $\upsilon, \nu \in \La(\vb_q)$. From Lemma \ref{lem:aproxwordsbelow} there are $J, J' \in \N$ such that $(\vb_{q^-_J}, \si)$ and $(\vb_{q^-_{J'}}, \si)$ are mixing subshifts of finite type and for every $m \geq J$, 
$$B_{|\upsilon|}(\vb_q) = B_{|\upsilon|}(\vb_{q^-_m}) = B_{|\upsilon|}(\vb_{q^-_J})$$ 
and for every $m \geq J$, 
$$B_{|\nu|}(\vb_q) = B_{|\nu|}(\vb_{q^-_m}) = B_{|\nu|}(\vb_{q^-_{J'}}).$$ 
Put $J'' = \max \set{J, J'}$. Since $(\vb_{q_{J''}}, \si)$ is a mixing subshift then there is $N \in \N$ such that for every $n \geq N$ there is $\om \in B_{n}(\vb_{q_{J''}})$ such that $\upsilon \om \nu \in \La(\vb_{q_{J''}})$. Since $\vb_{q_{J''}} \subset \vb_q$, the result follows.  
\end{proof}

\section{The specification property of $(\vb_q, \si)$}
\label{sec:topologicaldynamics}

\noindent In this section, we characterise the set of $q \in \vl$ such that $(\vb_q, \si)$ has the specification property. In order to do this, we introduce the notions of \emph{strongly irreducible} and \emph{weakly irreducible sequences}.

\begin{definition}\label{def:strongweaksequences}
We say that an irreducible sequence $\al = \al(q) \in \vb$ is \emph{strongly irreducible} if there exists $N \in \N$ such that for all $m \geq N$ one has $q^-_m \in \I$. We also say that an irreducible sequence is \emph{weakly irreducible} if there are infinitely many $m \in \mathbb{N}$ such that $q^-_m \notin \I$.  
\end{definition}

In a similar fashion we introduce the notion of \emph{strongly irreducible number}.

\begin{definition}\label{def:strongweaknumber}
A number $q \in \I$ is called \emph{strongly irreducible} if $\al(q)$ is strongly irreducible, similarly $q \in \I$ \emph{weakly irreducible} if $\al(q)$ is weakly irreducible. 
\end{definition}
We will use the notations
$$\SI = \set{q \in \ul : q \text{ is strongly irreducible}}$$  and  
$$\WI = \set{q \in \ul : q \text{ is weakly irreducible}}.$$ 
Clearly 
$$\WI = \I \setminus \SI \quad\text{ and }\quad \I, \WI \subsetneq \overline{\B} \cap [q_T, M+1].$$

In this section we will show some properties of the set of strongly and weakly irreducible sequences. Firstly, we mention that there are three different kinds of strongly irreducible sequences. 

\begin{definition}\label{def:typestrongseq}
We say that a strongly irreducible sequence $\al(q)$ is:
\begin{enumerate}[$i)$]
\item of \emph{Type 1} if for every $m \in \N$, $q^-_m \in \I$;  
\item of \emph{Type 2} if there is $N \in \N$ such that $q^-_m \in \I$ for every $m \geq N$ and $q^-_k < q_T$ for every $k < N$ ;
\item of \emph{Type 3} if there exists an $N \in \N$ and $m < N$ such that $q^-_m \notin \I$ with $q^-_m > q_T$ and $q^-_k \in \I$ for all $k \geq N$.
\end{enumerate}
\end{definition}

Let us illustrate Definition \ref{def:typestrongseq} with some examples.

\begin{example}
Consider $M = 1$, then:
\begin{enumerate}
\item Let $n \geq 3$. The sequence $(1^n0)^\f$ is strongly irreducible of type \emph{1}. Note that here $q^-_1 = q_G(1)$, so by \cite[Lemma 3.1]{AlcBakKon2016} $\al(q^-_1)$ is an irreducible sequence;
\item The sequence $(11010)^\f$ is strongly irreducible of type \emph{2}. Here $N = 4$ and the corresponding $n_N = 7$. The last non-irreducible sequence of the approximation from below is $q^-_3$ where $$\al(q^-_3) = (110100)^\f \lle \al(q_T(1));$$
\item The sequence $(1110010010)^\f$ is strongly irreducible of type \emph{3}. Here $N = 5$ and the last non-irreducible sequence of the 
approximation from below is $q^-_4$ where $$\al(q^-_4) = (111001000)^\f \lge \al(q_T(1)).$$
\end{enumerate}
Consider $M = 2$, then:
\begin{enumerate}
\item Let $n \geq 2$. The sequence $(2^n1)^\f$ is strongly irreducible of type $1$.
\item The sequence $(211211121111)^\f$ is strongly irreducible of type $2$. Here $N = 4$ and the last non-irreducible sequence of the approximation from below is $q^-_3$ where $\al(q^-_3) = (210)^\f \lle \al(q_T(2))$;
\item The sequence $(22010101)^\f$ is strongly irreducible of type $3$. Here $N = 6$ and the last non-irreducible sequence of the approximation from below is $q^-_5$ where $$\al(q^-_5) = (22010100)^\f \lge \al(q_T(2)).$$
\end{enumerate}
\end{example}

We can distinguish strongly irreducible numbers of types \emph{1, 2} and \emph{3} defining $q$ implicitly, as was done in Definition \ref{def:strongweaknumber}. 

To start our investigation, we want to show the reader why our intuition is that strongly irreducible sequences are, loosely speaking, the ``right ones'' to look for the specification property; i.e. a symmetric $q$-shift with the specification property is parametrised by a strongly irreducible sequence and vice versa. 

\begin{proposition}\label{prop:periodicstrong}
Set 
$$\operatorname{Per}(\I) = \set{q \in \I: \al(q) \text{ is periodic}}.$$ 
Then, $\operatorname{Per}(\I) \subset \SI$.
\end{proposition}
\begin{proof}
Let $q \in \operatorname{Per}(\I)$. Suppose that $\al(q) = (\al_1(q), \ldots \al_k(q))^\f$ has period $k$. Let $\set{q^-_m}_{m=1}^\f$ the natural approximation from below of $q$ and set 
\begin{equation}\label{eq:perstrong1}
\al(q^-_m) = (\al_1(q) \ldots \al_{n_m}(q)^-)^\f = (\al_1(q^-_m) \ldots \al_{n_m}(q^-_m))^\f   
\end{equation}
the quasi-greedy expansion of $q^-_m$ for every $m \in \N$. 

Since $q \in \I$ then for every $m \in \N$ we have
\begin{equation}\label{eq:irrecper}
\al_1(q) \ldots \al_{n_m}(q)(\overline{\al_1(q) \ldots \al_{n_m}(q)}^+)^\f = \al_1(q^-_m) \ldots \al_{n_m}(q^-_m)^+(\overline{\al_1(q^-_m) \ldots \al_{n_m}(q^-_m)})^\f\prec \al(q).
\end{equation} 
Also, since $q \in \operatorname{Per}(\I) \subset \I$ then Lemma \ref{lem:infinitelymanyirreducibles} implies that there are infinitely many $m \in \N$ such that $q^-_m \in \I$. 

Let
\[
N=\min\{m \in \N : q^-_m \in \I \text{ and } n_m > k\},
\]
i.e. $N$ is the first irreducible element of the natural approximation from below such that $\al(q^-_m)$ has larger period than the period of $\al(q)$. Observe that Lemma \ref{lem:infinitelymanyirreducibles} implies that $N$ is well-defined.  

We claim that for every $m \geq N$, $q^-_m \in \I$, i.e. $\al(q^-_m) = (\al_1(q) \ldots \al_{n_m}(q)^-)^\f$ is an irreducible sequence. We will proceed by induction. From the definition of $N$ we have that $q^-_N \in \I$. Suppose that for every $N \leq j \leq m$ we have that $q^-_j \in \I$. We show now that $q^-_{m+1} \in \I$. Clearly, $n_{m} < n_{m+1}$. We show in the following two cases that for every $i \in \N$ such that $(\al_1(q^-_{m+1}) \ldots \al_{i}(q^-_{m+1})^-)^\f \in \vb$ then 
\begin{equation}\label{eq:irrecper1}
\al_1(q^-_{m+1}) \ldots \al_{i}(q^-_{m+1})(\overline{\al_1(q^-_{m+1}) \ldots \al_{i}(q^-_{m+1})}^+)^\f \prec \al(q^-_{m+1}). 
\end{equation}
 
\begin{itemize}
\setlength\itemsep{1em}
\item Assume firstly that $1 \leq i < n_m$. From the induction hypothesis $q^-_m \in \I$ then using \eqref{eq:irrecper} we have 
\begin{align*}
\al_1(q^-_{m+1}) \ldots \al_i(q^-_{m+1})(\overline{\al_1(q^-_{m+1}) \ldots \al_i(q^-_{m+1})}^+)^\f &\prec \al_1(q^-_{m}) \ldots \al_i(q^-_{m})(\overline{\al_1(q^-_{m}) \ldots \al_i(q^-_{m})}^+)^\f \\
&\prec \al(q^-_m) \prec \al(q^-_{m+1}).
\end{align*}
So, \eqref{eq:irrecper1} holds in this case.

\item Consider now $i \geq n_m$. Observe that 
\[
\al_1(q) \ldots \al_{i}(q) = \al_1(q^-_{m+1}) \ldots \al_{n_{m}}(q^-_{m+1})\ldots \al_{i}(q^-_{m+1}) \in \La(\vb_{q^-_{m+1}}),   
\]
and  
\begin{equation}\label{eq:irredper2}
(\al_1(q^-_{m+1}) \ldots \al_{i}(q^-_{m+1})^-)^\f \quad \text{and} \quad (\overline{\al_1(q^-_{m+1}) \ldots \al_{i}(q^-_{m+1})}^+)^\f \in \vb_{q^-_{m+1}}.
\end{equation}
We claim that for every $\ell \in \N$ 
\begin{equation}\label{eq:irredper3}
\overline{\al(q^-_{m+1})} \prec \si^\ell((\overline{\al_1(q^-_{m+1}) \ldots \al_{i}(q^-_{m+1})}^+)^\f) \prec \al(q^-_{m+1}).
\end{equation}
Let us show \eqref{eq:irredper3}. From \eqref{eq:irredper2} we have
\[
\overline{\al(q^-_{m+1})} \lle \si^\ell((\overline{\al_1(q^-_{m+1}) \ldots \al_{i}(q^-_{m+1})}^+)^\f) \lle \al(q^-_{m+1}).
\]
for every $\ell \in \N$. Now, clearly $(\al_1(q^-_{m+1}) \ldots \al_{i}(q^-_{m+1})^-)^\f \prec \al_{q^-_{m+1}}$. Since $$(\al_1(q^-_{m+1}) \ldots \al_{i}(q^-_{m+1})^-)^\f \in \vb,$$ then Lemma \ref{lem:quasigreedyexpansion} implies there is a unique $p \in \vl$ such that $$\al(p) = (\al_1(q^-_{m+1}) \ldots \al_{i}(q^-_{m+1})^-)^\f \quad \text{and} \quad p \prec q^-_{m+1}.$$ Moreover $\vb_p \subsetneq \vb_q$. Therefore, for every $\ell \in \N$ 
\begin{align*}
\overline{\al(q^-_{m+1})} &\prec (\overline{\al_1(q^-_{m+1}) \ldots \al_{i}(q^-_{m+1})}^+)^\f\\
&\lle \si^\ell((\overline{\al_1(q^-_{m+1}) \ldots \al_{i}(q^-_{m+1})}^+)^\f) \lle (\overline{\al_1(q^-_{m+1}) \ldots \al_{i}(q^-_{m+1})}^+)^\f \\
&\prec \al(q^-_{m+1}).
\end{align*} 
So \eqref{eq:irredper3} holds. Then \eqref{eq:irredper3} implies $$(\overline{\al_1(q^-_{m+1}) \ldots \al_{i}(q^-_{m+1})}^+)^\f = (\overline{\al_1(q)\ldots \al_{n_{m}}(q) \ldots \al_{i}(q)}^+)^\f  \prec \sigma^{i}(\al(q^-_{m+1}))$$ which implies that \eqref{eq:irrecper1} holds in this case.
\end{itemize}
\end{proof}

We now show that transitive sofic symmetric $q$-shifts are parametrised by strongly irreducible numbers. 

\begin{proposition}\label{prop:eventualperiodicstrong}
Let $q \in \I$. Then, if $\al(q)$ is eventually periodic then $q \in \SI$.
\end{proposition}
\begin{proof}
Let $q \in \I$. Suppose that $\al(q) = \al_1(q) \ldots \al_r(q)(\al_{r+1}(q) \ldots \al_n(q))^\f$. Set $k$ to be the period of $\si^r(\al(q))$. Since $q \in \I$ then for every $j \in \N$ such that $(\al_1(q) \ldots \al_{j}(q)^-)^\f \in \vb$ then 
\begin{equation}\label{eq:irrecevper}
\al_1(q) \ldots \al_{j}(q)(\overline{\al_1(q) \ldots \al_{j}(q)}^+)^\f \prec \al(q).
\end{equation}
Let $\set{q^-_m}_{m=1}^\f$ the natural approximation from below of $q$. We will define $N$ in a similar way as in Proposition \ref{prop:periodicstrong}. Let 
\[
N = \mathop{\min} \set{ m \in \N: q^-_m \in \I \text{ and the period of }\al(q^-_m) > r + k }.
\] 
Using Lemma \ref{lem:infinitelymanyirreducibles} we get that $N$ is well defined. The argument to show that for every $j > N$, $q^-_j \in \I$ is exactly the same as in Proposition \ref{prop:periodicstrong}. 
\end{proof}

\subsection{Existence of non periodic strongly irreducible sequences and weakly irreducible sequences.}

We now show that, for a fixed $M \in \N$, $\I$ contains non-periodic and non-eventually periodic sequences, which can be either strongly or weakly irreducible.

We start our investigation by recalling that in \cite[Theorem 1.6]{KomKonLi2015} the following class of subsets was constructed. Let $M \in \N$ and fix $N \geq 2$. 

We consider a subset of $\vl$ that satisfy
\begin{equation}\label{eq:0NMN}
0^N\prec \al_{(r \cdot N)+1}(q)\ldots \al_{(r+1)\cdot N}(q)\prec M^N,\text{ for all } r\geq 2,
\end{equation}
namely
\begin{equation}\label{IN}
\I_N=\{
q\in\vl :
\al_1(q)\ldots \al_{2N}(q) = M^{2N-1}0 \text{ and \eqref{eq:0NMN} holds}
\}
\end{equation}

In \cite[Lemma 6.2]{AlcBakKon2016} it is shown that $\I_N \subset \I$ for every $N \geq 2$. Moreover $\I_N \subset \ul$. Then, Lemma \ref{lem:uandclosureu} $i)$ implies that for every $q \in \I_N$, $\al(q)$ is not periodic. Furthermore, in \cite[Lemma 6.4]{AlcBakKon2016} it is shown that $\dim_{\operatorname{H}}(\I_N) > 0$ for every $N \geq 2$, so $\I_N$ is uncountable. 

\begin{proposition}\label{pr:spec3}
Let $M \in \N$. Then, for every $N \geq 2$, $\I_N \subset \SI$.
\end{proposition}
\begin{proof}
Let $q \in \I_N$. Since $q \in \I$ then from Lemma \ref{lem:infinitelymanyirreducibles} we have that 
\begin{equation}\label{eq:IN1}
K = \min\set{k \in \N: (\al_1(q) \ldots \al_{n_k}(q)^-)^\f \in \I \text{ and } n_k \geq 3N}
\end{equation}
is well-defined. We claim that for every $j \geq K$ the sequence $(\al_1(q) \ldots \al_{n_j}(q)^-)^\f$ is irreducible. Fix $j \geq K$ and let $q^-_j$ be the corresponding element of the natural approximation from below. Then 
$$\al(q^-_j) = (\al_1(q^-_j) \ldots \al_{n_j}(q^-_j))^\f = (\al_1(q) \ldots \al_{n_j}(q)^-)^\f.$$  
Since $N \geq 2$ and $\al(q)$ satisfies \eqref{IN} then the word $0^{2N-1}$ is not a factor of $\al(q^-_j)$. 

We split the proof in two cases:
\begin{itemize}
\item Suppose that $i \in \N$ satisfies $(\al_1(q^-_j)\ldots \al_{n_i}(q^-_j)^-)^\f \in \vb$ and $1 \leq n_i \leq 2N$. Then 
$$\al_1(q^-_j)\ldots \al_{n_i}(q^-_j)(\overline{\al_1(q^-_j)\ldots \al_{n_i}(q^-_j)}^+)^\f  = M^{n_i}(0^{n_i-1}1)^\f.$$ 
Since $1 \leq n_i \leq 2N$ then $$M^{n_i}(0^{n_i-1}1)^\f \prec (\al_1(q^-_j) \ldots \al_{n_j}(q^-_j))^\f.$$ 

\item Suppose that $i \in \N$ satisfies that $(\al_1(q^-_j)\ldots \al_{n_i}(q^-_j)^-)^\f \in \vb$ and $2N \leq n_i$. Then
\begin{multline*}
\al_1(q^-_j)\ldots \al_{n_i}(q^-_j)(\overline{\al_1(q^-_j)\ldots \al_{n_i}(q^-_j)}^+)^\f = \\
\al_1(q^-_j)\ldots \al_{n_i}(q^-_j)(0^{2N-1}M\overline{\al_{2N+1}(q) \ldots \al_{n_i}(q^-_j)}^+)^\f.
\end{multline*}
Since $0^{2N-1}$ is not a factor of $\al(q^-_j)$ we obtain $$\al_{1}(q^-_j)\ldots \al_{n_i}(q^-_j)(0^{2N-1}M\overline{\al_{2N+1}(q) \ldots \al_{n_i}(q^-_j)}^+)^\f \prec \al(q^-_j)$$
and the result follows.
\end{itemize}
\end{proof}

\subsection{Construction of strongly irreducible sequences}
\label{constructstrong}
We describe a construction to obtain non-periodic strongly irreducible sequences. This construction will allow us to find strongly irreducible sequences in $\I \setminus \mathop{\bigcup}\limits_{N \geq 2} \I_N$. 

Recently, Allaart in \cite[Definition 2.1]{All2018} introduced the notion of \emph{fundamental word}. Given $M \in \N$, a word $\om = w_1 \ldots w_n \in \La(\Sig)$ with $n > 2$ is said to be a \emph{fundamental word} if 
\begin{equation}\label{eq:admissible1}
\overline{w_1 \ldots w_{n-i}} \lle w_{i+1} \ldots w_n \prec w_1 \ldots w_{n-i} \text{ for every }1 \leq i \leq n-1.
\end{equation}  
It is clear that if $\al = \al_1 \ldots \al_m$ is a fundamental word then $(\al_1 \ldots \al_m)^\f \in \vb$ \cite[p.6510]{All2018}. 

\begin{lemma}\label{lem:admissible1}
Let $q \in \ul$. Then, there exists a strictly increasing sequence $\set{m_j}_{j=1}^\f \in \N$ such that $\al_1(q) \ldots \al_{m_j}(q)$ is a fundamental word. 
\end{lemma}

\begin{remark}\label{rem:admissible1}
We want to make clear that the statement of Lemma \ref{lem:admissible1} remains valid for $q \in \vl$ making the following modifications. From Theorem \ref{thm:uclosureuv} $iv)$ we have that for every $q \in \vl \setminus \ul$, $\al(q)$ is a periodic sequence.  If $\al(q) = s^\f$ with $s \in \set{k+1, \ldots, M}$ given that $M = 2k+1$ or $s \in \set{k, \ldots, M}$ if $M = 2k$, then no factor of $\al(q)$ is a fundamental word, however $s^\f \in \overline{\ul}$ and $(\al_1(q) \ldots \al_m(q))^\f = s^\f$ for every $m \in \N$. Also, if $q \in \vl \setminus \ul$ with $\al(q) \neq s^\f$ as in the former case then $\al(q) = (\al_1(q) \ldots \al_m(q))^\f$ with $m \neq 1$. It is clear that the periodic block $\al_1(q) \ldots \al_m(q)$ is a fundamental word. Then, the sequence $m_j = m \cdot j$ for each $j \in \N$ satisfies the consequence of Lemma \ref{lem:admissible1}. 
\end{remark}

\begin{proof}[Proof of Lemma \ref{lem:admissible1}]
Let $q \in \ul$. Then $\overline{\al(q)} \prec \si^n(\al(q)) \prec \al(q)$ for every $n \in \N$. Let 
$$m_1 = \mathop{\min}\set{m \in \N: \al_m(q) < \al_1(q)}.$$
Since $q \in \ul$ then $\al(q)$ is not a periodic sequence, so $m_1$ is well defined. Note that $$\overline{\al_1(q) \ldots \al_{m_1-i}(q)} \lle \al_{i+1}(q) \ldots \al_{m_1}(q) \prec \al_{1}(q) \ldots \al_{m_1-i}(q)$$ since $\al_1(q)\ldots \al_{m_1-i}(q) = \al_1(q)^{m_1-i}$ for every $i \in \set{1, \ldots, m_1-1}$. 

Clearly, $(\al_1(q)\ldots\al_{m_1}(q))^\f \succ \al(q).$ Let us set $(\al_1(q)\ldots\al_{m_1}(q))^\f = \al(q^+_1)$ where $q^+_1$ is defined implicitly. Then, there exists a minimal integer $\ell > m_1$ such that 
\[
\al_{\ell}(q) < \al_1(q) = \al_1(q^+_1),
\]
\[
\overline{\al_1(q^+_1) \ldots \al_{\ell-1}(q^+_1)} \prec \al_1(q) \ldots \al_{\ell-1}(q) = \al_1(q^+_1) \ldots \al_{\ell-1}(q^+_1)
\]
and 
\[
\overline{\al_{\ell}(q^+_1)} \leq \al_{\ell}(q) < \al_{\ell}(q^+_1).
\] 
So, let $m_2 = \ell$. Observe that 
$$\overline{\al_1(q) \ldots \al_{m_2-i}(q)} \lle \al_{i+1}(q)\ldots \al_{m_2}(q) \lle \al_1(q) \ldots \al_{m_2-i}(q)$$
holds for every $1 \leq i \leq m_2 -1$; this, the definition of $m_1$ and the minimality of $m_2$ give that 
$$\overline{\al_1(q) \ldots \al_{m_2-i}(q)} \lle \al_{i+1}(q)\ldots \al_{m_2}(q) \prec \al_1(q) \ldots \al_{m_2-i}(q)$$ 
for every $1 \leq i \leq m_2 -1$. Thus, the word $\al_1(q)\ldots\al_{m_2}(q)$ is fundamental and $$\al(q^+_1) \succ (\al_1(q) \ldots \al_{m_2}(q))^\f \succ \al(q).$$ Let us call $\al(q^+_2) = (\al_1(q)\ldots\al_{m_2}(q))^\f$. 

Now we proceed by induction. Suppose that $n \in \N$ and that $m_1 \ldots m_n$ and $q^+_1 \ldots q^+_n$ have been already defined. Then we define $m_{n+1}$ to be the smallest integer $\ell > m_n$ such that $\al_{\ell}(q) < \al_1(q)$, 
$$\overline{\al_1(q^+_n) \ldots \al_{\ell-1}(q^+_n)} \prec \al_1(q) \ldots \al_{\ell-1}(q) = \al_1(q^+_n) \ldots \al_{\ell-1}(q^+_n)$$ 
and 
$$\overline{\al_{\ell}(q^+_n)} \leq \al_{\ell}(q) < \al_{\ell}(q^+_n).$$ 
We claim that $\al_1(q) \ldots \al_{m_{n+1}}(q)$ is a fundamental word. From the definition of $m_{n+1}$ we only need to show that for every $1 \leq i \leq m_{n+1}-1$,

\begin{equation}\label{eq:indadmissible1}
\overline{\al_1(q)\ldots \al_{m_{n+1}-i}(q)} \lle \al_i(q) \ldots \al_{m_{n+1}}(q) \prec \al_1(q)\ldots \al_{m_{n+1}-i}(q).
\end{equation}
From the induction hypothesis \eqref{eq:indadmissible1} is valid for $1 \leq i < m_n-1$ and from the minimality of $m_{n+1}$ we obtain that 
\eqref{eq:indadmissible1} holds for $m_{n} \leq i \leq m_{n+1} -1$. 
\end{proof}

In comparison with Definition \ref{def:approximations}, in Lemma \ref{lem:admissible1} we constructed another ``natural'' approximation of $q \in \vl$. It is possible to show that the sequence $\set{q^+_n}_{n=1}^\f$ is strictly decreasing and $q^+_n \mathop{\searrow}\limits_{n \to \f} q$. We call this the \emph{natural approximation from above of $q$}. 

\begin{lemma}\label{lem:admissible2}
If $q \in \overline{\B}$ then for every $j \in \N$ such that $\al_1(q) \ldots \al_{m_j}(q)$ is a fundamental word, the sequence $(\al_1(q) \ldots \al_{m_j}(q))^\f$ is irreducible or $*$-irreducible. 
\end{lemma}
\begin{proof}
Let $q \in \overline{\B}$. Let us assume that $q \geq q_T$. The proof for $q < q_T$ follows from a similar argument. In \cite[Lemma 2.6]{KalKonLiLu2017} (see also \cite[Theorem 3]{AlcBakKon2016}) it was shown that $\overline{\B} \subsetneq \overline{\ul}$. Also, from \cite[Lemma 4.10]{AlcBakKon2016} we have that if $q \in \overline{\B}$ satisfies that $\al(q)$ is periodic then $\al(q) \in \overline{\ul} \setminus \ul$. Then, the sequence given in Remark \ref{rem:admissible1} satisfies that $(\al_1(q) \ldots \al_{m_j}(q))^\f$ is irreducible for every $j \in \N$. So, let us assume that $\al(q)$ is not periodic. Fix $j \in \N$ such that $(\al_1(q) \ldots \al_{m_j}(q))^\f \in \vb$ and set $\al(q^+_j) = (\al_1(q) \ldots \al_{m_j}(q))^\f$. We want to show that if $i \in \N$ is such that $(\al_1(q)\ldots\al_i(q^+_j)^-)^\f \in \vb$ then 
\begin{equation}\label{eq:admissible2}
\al_1(q)\ldots\al_i(q^+_j)(\overline{\al_1(q)\ldots\al_i(q^+_j)}^+)^\f \prec \al(q^+_j).
\end{equation}
Suppose that there exists $i \in \N$ such that
$$\al_1(q)\ldots\al_i(q^+_j)(\overline{\al_1(q)\ldots\al_i(q^+_j)}^+)^\f \lge \al(q^+_j).$$ 
Without loss of generality we assume that such $i$ is minimal. Then, from \cite[Lemma 4.9 (1)]{AlcBakKon2016} there exists a unique $k < i$ such that $(\al_1(q^+_j)\ldots\al_k(q^+_j)^-)^\f$ is irreducible and 
\begin{align}\label{eq:admissible21}
(\al_1(q^+_j)\ldots\al_k(q^+_j)^-)^\f &\prec \al_1(q^+_j)\ldots\al_i(q^+_j)(\overline{\al_1(q)\ldots\al_i(q^+_j)}^+)^\f \\
&\prec \al_1(q^+_j)\ldots\al_k(q^+_j)(\overline{\al_1(q)\ldots\al_k(q^+_j)}^+)^\f. \nonumber       
\end{align}
Then \eqref{eq:admissible21} and Lemma \ref{lem:admissible1} imply that, whenever $n \geq j$, the sequence 
$$(\al_1(q^+_j)\ldots\al_k(q^+_j)^-)^\f \prec \al(q^+_n) \prec \al_1(q)\ldots\al_k(q^+_j)(\overline{\al_1(q)\ldots\al_k(q^+_j)}^+)^\f.$$ 
Thus, $q^+_n \in I(q^+_j)$ for every $n \geq j$, which is a contradiction to $q \in \overline{\B}$. 
\end{proof}

We now construct non-periodic strongly irreducible sequences. Fix $q_1 \in \operatorname{Per}(\I)$ with $\al(q_1) = (\al_1(q_1) \ldots \al_{m_1}(q_1))^\f$. Then $q_1$ parametrises an irreducible interval $I(q_1)$. Let $p_1$ be the right end point of $I(q_1)$. Then, $p_1 \in \overline{\B}\cap \ul$ and the quasi-greedy expansion of $p_1$ is given by
$$\al(p_1) = \al_1(q_1) \ldots \al_{m_1}(q_1)^+(\overline{\al_1(q_1) \ldots \al_{m_1}(q_1)})^\f.$$ 
Since $p_1 \in \overline{\B}$ then from Lemmas \ref{lem:admissible1} and \ref{lem:admissible2} there are infinitely many $m \in \N$ such that $\al_1(p_1) \ldots \al_{m}(p_1)$ is a fundamental word and $(\al_1(p_1) \ldots \al_{m}(p_1))^\f$ is an irreducible sequence. Taking  $m' = m_1+1$ it is clear that $\al_1(p_1) \ldots \al_{m'}(p_1)$ is a fundamental word, so there exists at least one $m \in \set{m_1+1 , \ldots, 2\cdot m_1}$ where $\al_1(p_1) \ldots \al_m(p_1)$ is a fundamental word. 

Let $m_2 \in \set{m_1+1 , \ldots, 2\cdot m_1}$ such that $\al_1(p_1)\ldots \al_{m_2}(p_1)$ is a fundamental word. Let $q_2$ be defined implicitly to be such that 
$$\al(q_2) = (\al_1(p_1)\ldots \al_{m_2}(p_1))^\f.$$ 
Thus, $q_2$ generates an irreducible interval $I(q_2)$ with right end point $p_2$. 

Set $m_3 \in \set{m_2+1 \ldots 2\cdot m_2}$ such that $\al_1(p_2) \ldots \al_{m_3}(p_2)$ is a fundamental word and $m_3 \neq m_2$. This generates $q_3$. So, let us assume that $q_1, \ldots, q_n$ and $p_1, \ldots, p_n$ has been already defined. Let $m_{n+1} \in \set{m_{n+1} \ldots 2\cdot m_n}$ such that $\al_1(p_n) \ldots \al_{m_{n+1}}(p_n)$ is a fundamental word and $m_{n+1} \neq m_{n}$, so we define $q_{n+1}$ implicitly to have quasi-greedy expansion 
$$\al(q_{n+1}) = (\al_1(p_n) \ldots \al_{m_{n+1}}(p_n))^\f.$$
Then $\set{q_n}_{n=1}^\f$ is a strictly increasing sequence of elements of $\I$. Moreover, $q_n < p^+_1$ where $p^+_1$ is given by the proof Lemma \ref{lem:admissible1} applied to $p_1$. Thus $q_n$ converges to a number $q\in (q_1, p^+_1)$. We claim that $q$ is strongly irreducible.

Let $\al(q)$ be the quasi-greedy expansion of $q$. From the above construction, it is clear that $q \in \overline{\B}$. To show that $\al(q)$ is an irreducible sequence, note that whenever $i \in \N$ satisfies that $(\al_1(q) \ldots \al_i(q)^-)^\f \in \vb$ then there exists $K \in \N$ such that for every $k \geq K$, 
$$\al_1(q) \ldots \al_i(q)(\overline{\al_1(q) \ldots \al_i(q)}^+)^\f \prec \al(q_k) \prec \al(q).$$ 
This implies the irreducibility of $\al(q)$. Let us show that $\al(q)$ is strongly irreducible. We claim that for every $i \geq m_1 +1$ such that 
\begin{equation}\label{eq:strong}
(\al_1(q)\ldots \al_i(q)^-)^\f \in \vb \quad (\al_1(q)\ldots \al_i(q)^-)^\f
\end{equation}
is an irreducible sequence.

Observe that for every $n \in \N$ there is no $i$, $m_n < i < m_{n+1}$, satisfying $(\al_1(q)\ldots \al_i(q)^-)^\f \in \vb$.
This holds since $\al_{m_n + 1}(q) \ldots \al_i(q)  = \al_1(q^+_{n-1}) \ldots \al_{i - m_{n-1}}(q^+_{n-1})$ for every $m_n < i < m_{n+1}.$ So, to show \eqref{eq:strong} we observe that if $i = m_n$ where $m_n$ is the period of $q_n$ for some $n \geq 2$ then $(\al_1(q) \ldots \al_{m_n}(q)^-)^\f = \al(q_n)$ which is an irreducible sequence. So, using $m_2$, we obtain that $\al(q)$ is strongly irreducible.

Observe that our construction depends entirely on the choice of $m_n$ at each step $n$. Note that at each step $n$ the number of choices of $m_{n+1}$ is strictly greater to the number of choices we have at the step $m_n$ since the period of $\al(q_n)$ is strictly smaller than the period of $\al(q_{n+1})$. Moreover, using Lemmas \ref{lem:admissible1} and \ref{lem:admissible2}, it is not difficult to check that for any $M \in \N$ one has $\I \setminus \I_N \neq \emptyset$ for every $N \geq 2$. We have then proved:

\begin{proposition}\label{pr:strongirreduciblesequence}
The set $\SI \setminus \mathop{\bigcup}\limits_{N \geq 2}\I_N$ is uncountable.
\end{proposition}

\subsection{Construction of weakly irreducible sequences} 
\label{constructweak}
Let us now construct a weakly irreducible sequence. The idea behind this construction is similar to the one for strongly irreducible sequences. The difference between both constructions is that for strongly irreducible sequences, in each step $n$ we find an irreducible sequence ``relatively far" from each entropy plateau, whereas for the case of weakly irreducible sequences we will strive to be ``relatively close" to entropy plateaus at each step.

Fix $q_1 \in \I$ with $\al(q_1) = (\al_1(q_1) \ldots \al_{m_1}(q_1))^\f$. Then, $q_1$ parametrises an irreducible interval $I(q_1)$. Let $p_1$ the right end point of $I(q_1)$. Then, $p_1 \in \overline{\B}\cap \ul$ and the quasi-greedy expansion of $p_1$ is given by 
$$\al(p_1) = \al_1(q_1) \ldots \al_{m_1}(q_1)^+(\overline{\al_1(q_1) \ldots \al_{m_1}(q_1)})^\f.$$ 
Fix $N_1 \in \N$. Then Lemmas \ref{lem:admissible1} and \ref{lem:admissible2} imply that there exist infinitely many $m \in \N$ such that $\al_1(p_1) \ldots \al_{m}(p_1)$ is a fundamental word and $(\al_1(p_1) \ldots \al_{m}(p_1))^\f$ is an irreducible sequence. In particular, there must exists at least one $m \in \set{(N_1 + 1) \cdot m_1 , \ldots, (N_1 + 2)\cdot m_1}$ where $\al_1(p_1) \ldots \al_m(p_1)$ is a fundamental word. Let $m_2 \in \set{(N_1 + 1) \cdot m_1 , \ldots, (N_1 + 2)\cdot m_1}$ be such that $\al_1(p_1)\ldots \al_{m_2}(p_1)$ is a fundamental word. Let $q_2$ be defined implicitly to be such that $$\al(q_2) = (\al_1(p_1)\ldots \al_{m_2}(p_1))^\f.$$ Notice that for every $i = k \cdot m_1$ with $1 \leq k \leq N_1$.

\begin{multline}\label{eq:strogirred1}
\al_1(p_1) \ldots \al_i(p_1) =\\ \al_1(q_1)\ldots\al_{m_1}(q_2)^+(\overline{\al_1(q_1)\ldots\al_{m_1}(q)})^{k-1}\overline{\al_1(q_1)\ldots\al_{m_1}(q)}^- \prec \al(p_1).
\end{multline}

Then, the cardinality of the set of integers $i$ such that $(\al_1(q_2) \ldots \al_i(q_2)^-)^\f \in \vb$ and $(\al_1(q_2) \ldots \al_i(q_2)^-)^\f$ is not irreducible is at least $N_1$.

Note that $q_1 < q_2$. Now, recall that $q_2$ generates an irreducible interval $I(q_2)$ with right end point $p_2$. Fix $N_2 \in \N$. Then there exists $m_3 \in \set{(N_2 + 1) \cdot m_2, \ldots (N_2 +2)\cdot m_2}$ such that $\al_1(p_1)\ldots \al_{m_3}(p_3)$ is a fundamental word. We define $q_3$ implicitly to have quasi-greedy expansion $(\al_1(p_1)\ldots \al_{m_3}(p_3))^\f$. Now, note that for every $i = k \cdot m_2$ with $1 \leq k \leq N_2$ 

\begin{multline}\label{eq:strogirred2} 
\al_1(p_2) \ldots \al_i(p_2) =\\ \al_1(q_2)\ldots\al_{m_1}(q_2)^+(\overline{\al_1(q_3)\ldots\al_{m_2}(q_3)})^{k-1}\overline{\al_1(q_2)\ldots\al_{m_2}(q)}^- \prec \al(p_2).
\end{multline}

This implies that 
$$\set{i \in \set{1, \ldots, m_2} : (\al_1(q_3) \ldots \al_i(q_3)^-)^\f \in \vb \text{ and } (\al_1(q_2) \ldots \al_i(q_2)^-)^\f \text{ is not irreducible}}$$ has cardinality at least $N_2+N_1$.

So, let us assume that $q_1, \ldots, q_n$, $p_1, \ldots, p_n$ and $N_1, \ldots, N_n$ have been already defined. Fix $N_{n+1}$ and let $m_{n+1} \in \set{(N_n + 1) \cdot m_{n} \ldots (N_n +2) \cdot m_n}$ such that $\al_1(p_n) \ldots \al_{m_{n+1}}(p_n)$ is a fundamental word, so we define $q_{n+1}$ implicitly to have quasi-greedy expansion $$\al(q_{n+1}) = (\al_1(p_n) \ldots \al_{m_{n+1}}(p_n))^\f.$$ Then $\set{q_n}_{n=1}^\f$ is a strictly increasing sequence of elements of $\I$. Moreover, $q_n < p^+_1$ as in Lemma \ref{lem:admissible1} applied to $p_1$. Thus $q_n$ converges to a number $q\in (q_1, p^+_1)$. Observe that $q \in \I$ since for every $i \in \N$ with $(\al_1(q) \ldots \al_{i}(q)^-)^\f$ there is $k \in \N$ such that $$\al_1(q) \ldots \al_{i}(q)(\overline{\al_1(q) \ldots \al_{i}(q)}^+)^\f \prec \al(q_k) \prec \al(q).$$ Note that, for each $n \geq 2$, $q_n$ forces the set
$$
\set{i \in \set{1, \ldots, m_2} : (\al_1(q_n) \ldots \al_i(q_n)^-)^\f \in \vb \text{ and } (\al_1(q_n) \ldots \al_i(q_n)^-)^\f \text{ is not irreducible
}}$$ to have cardinality at least $\sum_{j=1}^{n-1}N_j$. So $q \in \WI$.
This finishes our construction of weakly irreducible subsequences.

We want to remark once more that our construction now depends entirely on the choice of $m_n$ and $N_n$ at each step $n$. Note that at each step $n$ the number of choices of $m_{n+1}$ is strictly greater to the number of choices we have at the step $m_n$ since the period of $\al(q_n)$ is strictly smaller than the period of $\al(q_{n+1})$. Moreover, the set of sequences $\set{N_n}_{n = 1}^\f$ with $N_n \in \N$ is uncountable. Then, the following proposition holds.    
 
\begin{proposition}\label{pr:weakirreduciblesequence}
The set $\WI$ is uncountable.
\end{proposition}

From Propositions \ref{pr:aproxconvergencebelow} and \ref{prop:periodicstrong}, and Lemma \ref{lem:infinitelymanyirreducibles} we obtain the following result.

\begin{proposition}\label{pr:irreducibledense}
The set of irreducible sequences $\I$ is dense in $\overline{\B} \cap [q_T, M+1]$. Moreover, $\SI$ is dense in $\overline{\B} \cap [q_T, M+1]$. 
\end{proposition}

We now show that $\WI$ is also a dense subset relative to $\overline{\B} \cap [q_T, M+1]$. 

\begin{proposition}\label{pr:denseweakirr}
The set $\WI$ is dense in $\overline{\B} \cap [q_T, M+1]$
\end{proposition}
\begin{proof}
Let $q \in (\overline{\B}\cap[q_T, M+1]) \setminus \WI$ with quasi-greedy expansion $\al(q)$. Fix $\varepsilon > 0$. Then, from Lemma \ref{lem:quasigreedyexpansion} there is $\delta > 0$ such that if $d(\al(p), \al(q)) < \delta$ then $|p-q| \leq \varepsilon$. Let $n \in \N$ be sufficiently large to satisfy $1/2^n \leq \delta/2$. From Proposition \ref{pr:irreducibledense} and Lemma \ref{lem:infinitelymanyirreducibles} there exists $q' \in \I$ such that $\al(q') = (\al_1(q') \ldots \al_m(q'))$ is periodic of period $m$, irreducible, $\al(q') \prec \al(q)$ and $0 < d(\al(q'), \al(q)) < 1/2^n$. Since $\al(q)$ is irreducible and $\al(q') \prec \al(q)$ we have  
$$\al_1(q') \ldots \al_m(q')^+(\overline{\al_1(q') \ldots \al_m(q')})^\f \prec \al(q).$$ 
Then, from Lemma \ref{lem:infinitelymanyirreducibles} and Lemma \ref{lem:admissible2} there exist $N \in \N$ and $i \in \set{1, \ldots, m}$ such that $$\al_1(q') \ldots \al_m(q')^+(\overline{\al_1(q') \ldots \al_m(q')})^N\overline{\al_1(q') \ldots \al_i(q')}$$ is a fundamental word and 
$$\al_1(q') \ldots \al_m(q')^+(\overline{\al_1(q') \ldots \al_m(q')})^N\overline{\al_1(q') \ldots \al_i(q')} \prec \al(q).$$ 
Let $q_1$ be defined implicitly by 
$$\al(q_1) = (\al_1(q') \ldots \al_m(q')^+(\overline{\al_1(q') \ldots \al_m(q')})^N\overline{\al_1(q') \ldots \al_i(q')})^\f.$$
Then, appliying \ref{constructweak} to $q_1$ we obtain a weakly irreducible sequence $p$ such that 
$$0 < d(\al(p), \al(q)) < 1/2^n < \delta/2.$$
So, $d(p,q) < \varepsilon$. This shows that $\WI$ is dense in $\overline{\B} \cap [q_T, M+1]$. 
\end{proof}

\subsection{Characterisation of the specification property for $(\vb_q, \si)$}

We proceed now to characterise the set of $q \in \vl$ such that $(\vb_q, \si)$ has the specification property. For this endeavour we would like to recall the definition of the specification property: $(\vb_q, \si)$ has \emph{the specification property} if there exists $S \in \mathbb{N}$ such that for any $\upsilon, \nu \in \La(\vb_q)$ there
exists $\om \in B_S(\vb_q)$ such that $\upsilon\om\nu \in \La(\vb_q)$. Recall that, given $q \in \vl$ with $q \geq q_T$ the sequence $s_n(\vb_q)$ is given by
\begin{equation}\label{eq:spec}
\begin{split}
s_n = s_n(\vb_q) = \inf \{k \in \N : &\hbox{\rm{ for every }} \upsilon, \nu \in B_n(\vb_q) \hbox{\rm{ there exists }} \omega \in B_k(\vb_q)\\ 
&\hbox{\rm{ such that }} \upsilon \omega \nu \in \La(\vb_q) \}.
\end{split}
\end{equation}
Also, $(\vb_q, \si)$ has the specification property if and only if $\mathop{\lim}\limits_{n \to \infty} s_n < \infty$.

We now study the properties of $s_n(\vb_q)$. Firstly, let us prove some technical lemmas. Let us recall that $\{q^-_j\}_{j=1}^\f$ stands for the natural approximation from below of $q$ and let  $\al(q^-_j)$ be as in Definition \ref{def:approximations}. Lemma \ref{lem:spec1} will allow us to show that weakly irreducible sequences do not have specification. To prove that assertion, we find a lower bound for $s_n$, for large values of $n$.

\begin{lemma}\label{lem:spec2}
Let $q \in \I$ be such that there is $j \in \N$ such that $q^-_j > q_T$ and
$$(\al_1(q^-_{1})\ldots \al_{n_j}(q^-_j))^\f = (\al_1(q) \ldots \al_{n_j}(q)^-)^\f$$ 
is not irreducible. Let $k$ and $n_k$ be given by Lemma \ref{lem:spec1}. Then there exists $N_1 \in \N$ such that $s_n(\vb_q) > N_1 \cdot n_k  -1 $ for every $n \geq n_k$. 
\end{lemma}
\begin{proof}
Let $q \in \I$. Since there is $j \in \N$ such that $q^-_j > q_T$ and $(\al_1(q) \ldots \al_{n_j}(q)^-)^\f$ is not irreducible, then there exists $N \in \N$ such that 
\begin{equation}\label{eq:N}
\begin{split}
\al_1(q) \ldots \al_{n_k\cdot (N + 1)-1}(q) &= \al_1(q^-_j) \ldots \al_{n_k\cdot (N + 1)-1}(q^-_j) \\
&= \al_1(q^-_k) \ldots \al_{n_k}(q^-_k)^+(\overline{\al_1(q^-_k) \ldots \al_{n_k}(q^-_k)})^{N-1}\overline{\al_1(q^-_k) \ldots \al_{n_k-1}(q^-_k)},
\end{split}
\end{equation} 
where $k$ and $m_k$ are given by Lemma \ref{lem:spec1}. Let $N_1$ be the maximal integer satisfiying \eqref{eq:N}. Note that $\al(q_j)$ is periodic and  $q^-_j$ parametrises the subshift of finite type $(\vb_{q^-_j}, \si)$ . Moreover $\vb_{q^-_j} \subset \vb_{q}$. This shows that $N_1$ is well defined. 

Let $n \geq n_k$ and consider $\upsilon = u_1 \ldots u_n$, and let $\nu = v_1 \ldots v_n \in \La(\vb_q)$ be such that 
$$u_{n_k+1} \ldots u_n = \al_1(q) \ldots \al_{n_k}(q)$$ 
and 
$$\overline{\al_1(q) \ldots \al_{n_k}(q)}^+ \prec v_1 \ldots v_{n_k} \prec \al_1(q) \ldots \al_{n_k}(q)^-.$$ 
Since $q \in \I$ then $(\vb_q, \si)$ is a transitive subshift. Then, there exists $\omega \in \La(\vb_q)$ such that $\upsilon \omega \nu \in \La(\vb_q)$. Also, from the choice of $\upsilon$ and $\nu$ we have  
$$u_{n_k+1} \ldots u_n \text{ and } v_1 \ldots v_n \in \La(\vb_{q^-_j}).$$ 
Moreover, since $(\vb_{q^-_j}, \si)$ is not a transitive subshift, then $\om \in \La(\vb_q) \setminus \La(\vb_{q^-_j})$. Note that the period of $\al(q^-_j) = N_1 \cdot n_k.$ From Lemma \ref{lem:aproxwordsbelow} we have that $|\om| \geq N_1 \cdot n_k -1$. 
\end{proof}

\begin{proposition}\label{pr:spec1}
If the subshift $(\vb_q, \si)$ has the specification property then $q \in \SI$.
\end{proposition}
\begin{proof}
We will show that $(\vb_q , \si)$ does not have the specification property if $q \notin \SI$. Firstly, note that if $q \in \vl \cap ((q_G, M] \setminus \I)$, $\al(q)$ is not an irreducible sequence. Then, from \cite[Theorem 1]{AlcBakKon2016} $(\vb_q, \si)$ is not transitive, thus
 $(\vb_q, \si)$ cannot have the specification property.  

Recall that if $(\vb_q, \si)$ is a transitive subshift then $q \in \I$. Moreover, if $q \notin \SI$ then $q \in \WI$, that is, there are infinitely many $j \in \N$ such that $q^-_j$ satisfies that $\al(q^-_j)$ is periodic but not irreducible. On the other hand, \cite[Proposition 6.1]{AlcBakKon2016} and Proposition \ref{pr:coded} imply that there are also infinitely many $m \in \N$ such that $\al(q^-_{m})$ is periodic and irreducible. This implies that there exist infinitely many $j \in \N$ such that $q^-_j$ and such that $\al(q^-_j)$ is not irreducible and $q^-_j \geq q_T$. Let $m_1 \in \N$ be such that $q^-_{m_1}$ is irreducible and $q^-_{m_1-1} \geq q_T$ and $q^-_{m_1-1}$ is not irreducible. Let $m_1$ be the period of $\al(q^-_{m_1})$. Let $k_1$ and $n_{k_1}$ be given by Lemma \ref{lem:spec1}. Then, from Lemma \ref{lem:spec2} we have that for every $n \geq m_1$, $s_{n}(\vb_q) \geq N_1 \cdot n_{k_1}$ for some $N_1 \in \N$. Since $q$ is weakly irreducible, then there exists $m_2 > m_1 + 1 $ such that $q^-_{m_2}$ is irreducible and $q^-_{m_2-1}$ is not irreducible. Since $\{q^-_j\}$ is an increasing sequence then $q_T < q_{m_1-1} < q_{m_1} < q_{m_2-1} < q_{m_2}$. Then again, Lemma \ref{lem:spec1} and Lemma \ref{lem:spec2} imply that there are $N_2 \in \N$ and $n_{k_2} \in \N$ such that for every $n \geq m_2$, $s_{m_2}(\vb_q) \geq N_2 \cdot n_{k_2} - 1$ for $N_2 \in \N$. Note that $n_{k_2} \geq n_{m_1} \geq N_1 \cdot n_{k_1} - 1$. This implies $s_{m_2}(\vb_q) > s_{m_1}(\vb_q)$.  Then for every $r \in \N$, it is clear that
$$s_n(\vb_q) > s_{m_r}(\vb_q) \geq N_r \cdot n_{k_r} - 1.$$ 
Also, $n_{k_r} \longrightarrow \infty$ as $n \to \infty$. Then $\mathop{\lim}\limits_{n\to \infty} s_n (\vb_q)$ is not bounded from above, which shows that $(\vb_q, \si)$ does not have the specification property.
\end{proof}

Showing that strongly irreducible sequences parametrise symmetric subshifts with the specification property is more complicated since it is necessary to find an upper bound for $s_n$. The following technical lemma will allow us to obtain such an upper bound. 

\begin{lemma}\label{lem:spec3}
Let $q \in \SI$. Then, there exists $j \in \N$ such that either  $\om = \al_1(q) \ldots \al_j(q)^-$ or $\overline{\om} = \overline{\al_1(q) \ldots \al_j(q)}^+$ satisfy $\om\nu$ or $\overline{\om}\nu$ are in $\La(\vb_q)$, for any $\nu \in \La(\vb_q)$.
\end{lemma}
\begin{proof}
Let $\nu = v_1 \ldots v_r \in \La(V_q)$. We consider three cases:

\noindent \emph{\textbf{Case 1}}: Suppose that $\al(q)$ is strongly irreducible of type $1$. Let us assume that $\overline{\al_1(q)} < v_1 \leq \al_1(q)$. Since $q > q_T$ one has $\overline{\al_1(q)} \leq \al_1(q)^- < \al_1(q)$, hence $$\overline{\al_1(q) \ldots \al_{r+1}(q)} \prec \al_1(q)^-v_1\ldots v_r \prec \al_1(q) \ldots \al_{r+1}(q).$$ Since $\nu \in \La(\vb_q)$ then our result holds. 

Now, suppose that $v_1 = \overline{\al_1(q)}$. Since $q > q_T$ it follows that $$\overline{\al_1(q)} < \overline{\al_1(q)}^+ \leq \al_1(q).$$  If $\overline{\al_1(q)}^+ < \al_1(q)$, using symmetry we obtain that
$$\overline{\al_1(q) \ldots \al_{r+1}(q)} \prec \overline{\al_1(q)}^+v_1 \ldots v_r \prec \al_1(q) \ldots \al_{r+1}(q),$$ and since $v_1 \ldots v_r \in \La(\vb_q)$ our result follows.   

Finally, we assume that $\overline{\al_1(q)}^+ = \al_1(q)$. Recall that $q > q_T$ and that $q \in \SI \subset \I$. Thus, $$\overline{\al_1(q)}^+v_1 = \al_1(q)\overline{\al_1(q)} \prec \al_1(q)\al_2(q).$$ Therefore, $$\overline{\al_1(q) \ldots \al_{r+1}(q)} \prec \overline{\al_1(q)}^+v_1 \ldots v_r \prec \al_1(q) \ldots \al_{r+1}(q).$$ As a result, the proposition follows. 

Note that $j = 1$.  

\noindent \emph{\textbf{Case 2}}: Suppose that $\al(q)$ is strongly irreducible of type $2$. Then, let 
$$k = \mathop{\max}\set{i \in \N: (\al_1(q)\ldots \al_i(q)^-)^\f \in \vb \text{ and } (\al_1(q)\ldots \al_i(q)^-)^\f \text{ is not irreducible}}.$$ 
Since $\al(q)$ is of type $2$, then $\al(q_T) \prec (\al_1(q)\ldots \al_k(q)^-)^\f \prec \al(q_T).$ This implies that 
\begin{align*}
\al_1(q) &= k+1 \quad\text{ if } M = 2k;\\
\al_1(q)\al_2(q) &= (k+1)(k+1) \quad \text{ if } M = 2k+1.
\end{align*}
We claim that $j = 1$ if $M = 2k$ and $j = 2$ if $M = 2k+1$. Let us prove firstly the case when $M = 2k$.  Since $\al(q)$ is strongly irreducible of type 2, then $\overline{\al_1(q)}^+ = \al_1(q)^- = k.$ Therefore, 
$$\overline{\al_1(q) \ldots \al_{r+1}(q)} \prec  k v_1 \ldots v_r \prec \al_1(q) \ldots \al_{r+1}(q),$$ 
which implies our result. 
Now, if $M = 2k+1$, since $\al(q)$ is strongly irreducible of type 2, then $\al_1(q_G)\al_2(q_G) = \al_1(q)\al_2(q)$. This implies that $\al(q) \in \set{k, k+1}^\f.$ Then, if $v_1 = \al_1(q)$ note that 
\begin{align}\label{con1}
\overline{\al_1(q)\ldots \al_{r+2}(q)}  &\prec k+1\, k v_1 \ldots v_r = \al_1(q) \al_2(q)^- v_1 \ldots v_r \prec \al_1(q)\ldots \al_{r+2}(q);\\
\overline{\al_1(q)\ldots \al_{r+1}(q)}  &\prec  k v_1 \ldots v_r = \al_2(q)^- v_1 \ldots v_r \prec \al_1(q)\ldots \al_{r+1}(q). \nonumber
\end{align}
Using symmetry, we have that if $v_1 = \overline{\al_1(q)}$ then 

\begin{equation}\label{con2}
\begin{split}
\overline{\al_1(q)\ldots \al_{r+2}(q)}  &\prec k\, k+1 v_1 \ldots v_r = \al_1(q) \al_2(q)^- v_1 \ldots v_r \prec \al_1(q)\ldots \al_{r+2}(q);\\
\overline{\al_1(q)\ldots \al_{r+1}(q)}  &\prec  k+1 v_1 \ldots v_r = \al_2(q)^- v_1 \ldots v_r \prec \al_1(q)\ldots \al_{r+1}(q). 
\end{split}
\end{equation} 
Then, \eqref{con1} and \eqref{con2} together with the fact that $\nu \in \La(\vb_q)$ imply that our result holds ---see also \cite[Lemma 3.9]{AlcBakKon2016}.

\noindent \emph{\textbf{Case 3}}: Suppose that $\al(q)$ is strongly irreducible of type $3$. Let 
$$k = \mathop{\max}\set{i \in \N: (\al_1(q)\ldots \al_i(q)^-)^\f \in \vb \text{ and } (\al_1(q)\ldots \al_i(q)^-)^\f \text{ is not irreducible}}.$$ 
Since $\al(q)$ is strongly irreducible of type \emph{3} then from Lemma \ref{lem:spec1} there exists a unique $j < k$ such that $\al_1(q) \ldots \al_j(q)^-$ is a fundamental word, $(\al_1(q) \ldots \al_j(q)^-)^\f$ is irreducible and 
$$(\al_1(q) \ldots \al_j(q)^-)^\f \prec (\al_1(q)\ldots \al_k(q)^-)^\f \prec \al_1(q) \ldots \al_j(q)(\overline{\al_1(q) \ldots \al_j(q)}^+)^\f \prec \al(q),$$
in particular 
\begin{equation}\label{con3}
\overline{\al_1(q) \ldots \al_j(q)} \prec \overline{\al_1(q) \ldots \al_j(q)}^+ \prec \al_1(q) \ldots \al_j(q)^- \prec \al_1(q) \ldots \al_j(q).
\end{equation}
If $\overline{\al_1(q)} \prec v_1 \prec \al_1(q)$ then, since $\nu \in \La(\vb_q)$, \eqref{con3} implies 
$$ \overline{\al_1(q) \ldots \al_j(q)}^+\nu, \quad \text{and}\quad \al_1(q) \ldots \al_j(q)^-\nu \in \La(\vb_q).$$
Suppose that $v_1 = \al_1(q)$, note that since $\al_1(q) \ldots \al_j(q)^-$ is a fundamental word then $\al_j(q)^- < \al_1(q) = v_1$. Therefore, from \eqref{con3} we obtain $$\overline{\al_1(q) \ldots \al_{j+r}(q)} \prec \al_1(q) \ldots \al_j(q)^-\nu \prec \al_1(q) \ldots \al_{j+r}(q).$$ 
Moreover, since $\al_1(q) \ldots \al_j(q)^-$ is a fundamental word and $\nu \in \La(\vb_q)$ we have that for every $2 \leq i \leq j$, 
\[
\overline{\al_1(q) \ldots \al_{j+r-i}(q)} \lle \al_{i+1}(q)\ldots\al_{j}(q)v_1 \ldots v_r \prec \al_1(q) \ldots \al_{j+r-i}(q),
\] 
which implies that $$\al_1(q) \ldots \al_j(q)^-\nu \in \La(\vb_q).$$
By symmetry (see \eqref{eq:lang1}), using a similar argument we obtain that if $v_1 = \overline{\al_1(q)}$ then $$\overline{\al_1(q) \ldots \al_j(q)}^+\nu \in \La(\vb_q).$$ 
\end{proof}

Note that the index $j$ obtained in Lemma \ref{lem:spec3} does not depend neither on $\nu \in \La(\vb_q)$ nor on the length of $\nu$. However, it depends on $q$. We now recall \cite[Lemma 3.8]{AlcBakKon2016}. We will include the proof since it will be used later on.

\begin{lemma}\label{lem:suffix}
Let $q \in \vl$. Then, for any word $\upsilon \in\La(\vb_q)$ and any $m>|\upsilon|$ there exists $\eta\in \La(\vb_q)$ such that $\upsilon\eta \in \La(\vb_q)$ and $\al_1(q)\ldots\al_m(q)$ or $\overline{\al_1(q)\ldots\al_m(q)}$ is a suffix of $\upsilon\eta$.
\end{lemma}   
\begin{proof}
Let $\upsilon = u_1 \ldots u_n \in \La(\vb_q)$ and let $m > n$ be fixed. Then, 
\begin{equation}\label{eq:suffix1}
\overline{\al_1(q) \ldots \al_{n-i}(q)} \lle u_{i+1} \ldots u_n \lle \al_1(q) \ldots \al_{n-i}(q) \text{ for every }i \in \set{0, 1, \ldots, n-1}. 
\end{equation}
If 
\begin{equation}\label{eq:suffix2}
\overline{\al_1(q) \ldots \al_{n-i}(q)} \prec u_{i+1} \ldots u_n \prec \al_1(q) \ldots \al_{n-i}(q) \text{ for every }i \in \set{0, 1, \ldots, n-1}, 
\end{equation}
then it is clear that the words $\eta = \al_1(q) \ldots \al_m(q)$ and $\eta' = \overline{\al_1(q) \ldots \al_m(q)}$ satisfy the conclusion of the lemma. 

Suppose now that 
\begin{equation}\label{eq:suffix3}
\overline{\al_1(q) \ldots \al_{n-i}(q)} \prec u_{i+1} \ldots u_n \lle \al_1(q) \ldots \al_{n-i}(q) \text{ for every }i \in \set{0, 1, \ldots, n-1}. 
\end{equation}

Then, we define
\begin{equation}\label{eq:suffix4}
s^+ = s^+(\upsilon) = \min\set{s \in \set{0,1,\ldots,n-1}: u_{s+1}\ldots u_n=\al_1(q)\ldots\al_{n-s}(q)}.
\end{equation} 
Then, the minimality of $s^+$ and \eqref{eq:suffix1} imply that \[
\overline{\al_1(q)\ldots\al_{n-i}(q)}\prec u_{i+1}\ldots u_n\prec\al_1(q)\ldots\al_{n-i}(q) \textrm{ for all }\quad 0\leq i<s^+.
\] 
Then, for every $j \in \N$ the word $\eta = \al_{n-s^++1}(q)\ldots\al_{n-s^+ + j}(q)$ satisfies the conclusion of the lemma. 

Now, let us assume that 
\begin{equation}\label{eq:suffix5}
\overline{\al_1(q) \ldots \al_{n-i}(q)} \lle u_{i+1} \ldots u_n \prec \al_1(q) \ldots \al_{n-i}(q) \text{ for every }i \in \set{0, 1, \ldots, n-1}, 
\end{equation}

In a similar way to \eqref{eq:suffix4} we define
\begin{equation}\label{eq:suffix6}
s^- = s^-(\upsilon) = \min\set{s \in \set{0,1,\ldots,n-1}: u_{s+1}\ldots u_n=\overline{\al_1(q)\ldots\al_{n-s}(q)}}.
\end{equation} 
Here, the minimality of $s^-$ and \eqref{eq:suffix1} imply that
\[
\overline{\al_1(q)\ldots\al_{n-i}(q)}\prec u_{i+1}\ldots u_n\prec\al_1(q)\ldots\al_{n-i}(q) \textrm{ for all }\quad 0\leq i<s^-.
\] 
So, for every $j \in \N$ the word $\eta = \overline{\al_{n-s^-+1}(q)\ldots\al_{n-s^- + j}(q)}$ satisfies the conclusion of the lemma.

Finally, suppose that 
\begin{equation}\label{eq:suffix7}
\overline{\al_1(q) \ldots \al_{n-i}(q)} \lle u_{i+1} \ldots u_n \lle \al_1(q) \ldots \al_{n-i}(q) \text{ for every }i \in \set{0, 1, \ldots, n-1}, 
\end{equation}

Using symmetry, we may assume without losing generality that $s^+ < s^-$, i.e 
\[
u_{s^++1}\ldots u_n=\al_1(q)\ldots\al_{n-s^+}(q).
\] 
If $s^+=0$, then $\upsilon = \al_1(q)\ldots\al_{n}(q)$. Therefore, for every $j \in \N$ the word $\eta = \al_{n+1}(q)\ldots\al_{n+j}(q)$ satisfies the conclusion of the lemma. Let us assume now that $s^+ \neq 0$. Then, the minimality of $s^+$, as well as \eqref{eq:suffix1} and the inequality $s^+ < s^-$ imply 
\[
\overline{\al_1(q)\ldots\al_{n-i}(q)}\prec u_{i+1}\ldots u_n\prec\al_1(q)\ldots\al_{n-i}(q) \textrm{ for all }\quad 0\leq i<s^+.
\] 
Then, for every $j \in \N$ the word $\eta = \al_{n-s^++1}(q)\ldots\al_{n-s^+ + j}(q)$ satisfies the conclusion of lemma. 
\end{proof}

\begin{proposition}\label{pr:spec2}
If $q \in \SI$ and there exists $K \in \N$ such that
$$d(\si^k(\al(q)), \overline{\al(q)}) \geq 1/2^K$$  
for every $k \in \N$ then $(\vb_q, \si)$ has the specification property.
\end{proposition}
\begin{proof}
Let $q \in \SI$. Then, there exists $N \in \N$ such that for every $m \geq N$, $q^-_m \in \I$. Let us assume that there is $K \in \N$ such that $d(\si^k(\al(q)), \overline{\al(q)}) \geq 1/2^K$ for every $k \in \N$. Clearly, $d(\si^k(\overline{\al(q)}), \al(q)) \geq 1/2^K$ for every $k \in \N$. Set $$\al(q^-_m) = (\al_1(q^-_m)\ldots \al_{n_m}(q^-_m))^\f = (\al_1(q)\ldots \al_{n_m}(q)^-)^\f$$ for every $m \in \N$. Let
\begin{equation}\label{spec1}
J = \min\set{m \in \N: \overline{\al_1(q) \ldots \al_K(q)}, \quad \al_1(q) \ldots \al_K(q) \in \La(\vb_{q^-_m}) \text{ and } q^-_m \in \I}.
\end{equation}
It is easy to check that $n_{J} > K$.
 
\vspace{1em}
\noindent \textsl{Claim A}: For every $m \geq J$ we have that $n_{m+1} - n_m \leq K+1.$ 

To show this, suppose that there is $m \geq J$ such that $n_{m+1} - n_m > K+1$. Then, for every $n_m+1\leq j \leq n_{m+1}$, $$(\al_1(q)\ldots\al_j(q)^-)^\f \notin \vb.$$ This, combined with the assumption $n_{m+1} - n_m > K+1$, implies that
$$\al_{n_m+1}(q)\ldots \al_{n_m+K+1}(q) = \al_{n_m+1}(q)\ldots \al_{n_{m+1}-i}(q) = \overline{\al_1(q) \ldots \al_{K+1}(q)}$$
for some $i \geq 0.$ Then, $$d(\si^{n_m}(\al(q)),\overline{\al(q)}) \leq 1/2^{K+1}$$ 
which is a contradiction. Therefore, our claim is true.

\vspace{1em}Let $\upsilon = u_1\ldots u_\ell$ and $\nu = v_1 \ldots v_\ell \in B_\ell(\vb_q)$. From Lemma \ref{lem:aproxwordsbelow} there exists $L\in \N$ such that $\upsilon, \, \nu \in \La(\vb_{q^-_m})$ for every $m \geq L$. Consider 
\begin{equation}\label{eq:spec2}
J' = \min\set{m > \max\set{L, N, J}: \upsilon, \nu \in \La(\vb_{q^-_m})}. 
\end{equation}
Note that Proposition \ref{pr:mixingsft} $(\vb_{q^-_{J}}, \si)$ and $(\vb_{q^-_{J'}}, \si)$ are mixing subshifts of finite type. Also, 
\begin{equation}\label{eq:spec4}
J \leq N \text{ and }\quad K < n_{J} \leq n_{N} \leq n_{J'}.
\end{equation}
Moreover, if $\ell \leq m_N$ then $J' = N+1$. Since $\upsilon \in B_{\ell}(\vb_{q^-_{J'}})$, then 
\begin{equation}\label{eq:spec3}
\overline{\al_1(q^-_{J'}) \ldots \al_{\ell-i}(q^-_{J'})} \lle u_{i+1} \ldots u_\ell \lle \al_1(q^-_{J'}) \ldots \al_{\ell-i}(q^-_{J'}).
\end{equation}

\vspace{1em}We now split the proof in three cases:

\noindent \emph{\textbf{Case 1}}: Strict inequalities hold in \eqref{eq:spec3}. 

Then from Lemma \ref{lem:suffix}, we have that the words $$\al_1(q^-_{J'}) \ldots \al_{t}(q^-_{J'})\quad \text{ and }\quad \overline{\al_1(q^-_{J'}) \ldots \al_{t}(q^-_{J'})}$$ satisfy that $$\upsilon \al_1(q^-_{J'})\ldots \al_{t}(q^-_{J'}), \quad \upsilon\overline{\al_1(q^-_{J'})\ldots \al_{t}(q^-_{J'})} \in \La(\vb_{q^-_{J'}})$$ for every $t \in \N$. Since $J < J'$, Lemma \ref{lem:aproxwordsbelow} implies that $$\upsilon \al_1(q)\ldots \al_{n_{J}}(q)^- = \upsilon \al_1(q^-_{J'}) \ldots \al_{n_M}(q^-_{J'})^- $$ and $$ \upsilon \overline{\al_1(q)\ldots \al_{n_{J}}(q)}^+ = \upsilon \overline{\al_1(q^-_{J'}) \ldots \al_{n_{J}}(q^-_{J'}) }^+ \in \La(\vb_{q^-_{J'}}).$$ From Lemma \ref{lem:spec3} there is $j \in \N$ such that $$\al_1(q)\ldots\al_j(q)^-\nu \quad \text{ or }\quad \overline{\al_1(q)\ldots\al_j(q)}^+\nu \in \La(\vb_{q_{J'}}).$$ Then, either $$\omega = \al_1(q^-_{J'}) \ldots \al_{m_M}(q^-_{J'})^-{\al_1(q)\ldots\al_j(q)}^-$$ or $$\omega =  \overline{\al_1(q^-_{J'}) \ldots \al_{m_{J}}(q^-_{J'}) }^+\overline{\al_1(q)\ldots\al_j(q)}^+ \in \La(\vb_{q^-_{J'}})$$ satisfy that $\upsilon \omega \nu \in \La(\vb_{q^-_{J'}})$, so $\upsilon \omega \nu \in \La(\vb_{q})$. Since $J < J'$ we obtain that $|\om| = n_{J} +j$.     

\noindent \emph{\textbf{Case 2}}: Let $s^+, s^-$ given by Lemma \ref{lem:suffix} and let $s = \min\set{s^+, s^-}$. We prove the case when $s = s^+$ since the proof for other case is analogous.

\emph{Case 2 $a)$}: Suppose that $s = 0$.  Then, there is $N_1 \in \N$ such that either 
\begin{equation}\label{eq:speccase1}
\upsilon = (\al_1(q^-_{J'}) \ldots \al_{n_{M'}}(q^-_{J'}))^{N_1}
\end{equation} 
or 
\begin{equation}\label{eq:speccase2}
\upsilon = (\al_1(q^-_{J'}) \ldots \al_{n_{J'}}(q))^{N_1} \al_{1}(q^-_{J'}) \ldots \al_{l}(q^-_{J'}) \text{ for } l \in \set{1, \ldots, n_{J'}-1}. 
\end{equation} 
Suppose that \eqref{eq:speccase1} holds. Again, as a consequence of Lemma \ref{lem:suffix}, we obtain that 
$$\al_1(q^-_{J'}) \ldots \al_{t}(q^-_{J'})\quad \text{ and }\quad \overline{\al_1(q^-_{J'}) \ldots \al_{t}(q^-_{J'})}$$ 
satisfy 
$$\upsilon \al_1(q^-_{J'})\ldots \al_{t}(q^-_{J'}), \quad \upsilon\overline{\al_1(q^-_{J'})\ldots \al_{t}(q^-_{J'})} \in \La(\vb_{q^-_{J'}})$$
for every $t \in \N$. Using a similar argument as in Case $1)$, we have that either 
$$\omega = \al_1(q^-_{J'}) \ldots \al_{n_J}(q^-_{J'})^-{\al_1(q)\ldots\al_j(q)}^-$$
or 
$$\omega =  \overline{\al_1(q^-_{J'}) \ldots \al_{n_{J}}(q^-_{J'}) }^+\overline{\al_1(q)\ldots\al_j(q)}^+ \in \La(\vb_{q^-_{J'}})$$
satisfy  
$\upsilon \omega \nu \in \La(\vb_{q^-_{J'}})$, and $\upsilon \omega \nu \in \La(\vb_{q})$. In this case, we also have $|\om| = n_{J} +j$.     

Suppose now that \eqref{eq:speccase2} holds. Then $l \in \set{1, \ldots n_{J'}-1}$. If $l \in \set{1, \ldots, N-1}$ then $\upsilon\al(q)_{l+1} \ldots \al_{n_N}(q)^- \in \La(\vb_{q^-_{J'}})$. Then either 
$$\omega = \al(q)_{l+1} \ldots \al_{n_N}(q)^-\al_1(q)\ldots\al_j(q)^-$$ 
or 
$$\om = \al(q)_{l+1} \ldots \al_{n_N}(q)^-\overline{\al_1(q)\ldots\al_j(q)}^+ $$ 
satisfy that $\upsilon \omega \nu \in \La(\vb_{q_{J'}}) \subset \La(\vb_{q})$. Note that $|\omega| \leq n_N-1 + j$. If $l \in \set{m_N \ldots n_{J'} - 1}$, Claim \textsl{A} together with \eqref{eq:spec4} imply that there is $l' \in \set{m_N \ldots m_{J'} - 1}$ such that $l < l'$, $(\al_1(q) \ldots \al_{l}(q)^-)^\f$ is an irreducible sequence and $l' \leq l \leq K+1$. Then, using Lemma \ref{lem:spec3} we have that $\omega = \al_{n_N+1}(q) \ldots \al_{l'}(q)\al_1(q)\ldots \al_j(q)^-$ or $\omega = \al_{n_N+1}(q) \ldots \overline{\al_{l'}(q)\al_1(q)\ldots \al_j(q)}^+$ satisfies that $\upsilon \om \nu \in \La(\vb_{q_{N''}}) \subset \La(\vb_{q}).$ Note that $|\om| = l' - l +j \leq K+1 + j$.

\emph{Case 2 $b)$}: Finally, let us assume that $s \neq 0$. Then
\begin{equation}\label{eq:sneq0}
u_{s+1} \ldots u_{\ell} = \al_1(q^-_{J'}) \ldots \al_{\ell-s}(q^-_{J'}).  
\end{equation}
Then, there exists $N_2 \in \N$ such that 
\begin{equation}\label{eq:speccase3}
u_{s+1} \ldots u_{\ell} = (\al_1(q^-_{J'}) \ldots \al_{n_{J'}}(q^-_{J'}))^{N_2}
\end{equation} 
or 
\begin{equation}\label{eq:speccase4}
u_{s+1} \ldots u_{\ell} = (\al_1(q^-_{J'}) \ldots \al_{n_{J'}}(q))^{N_2} \al_{1}(q^-_{J'}) \ldots \al_{l}(q^-_{J'}) 
\end{equation} 
for $l \in \set{1, \ldots, m_{J'}-1}$.
Then, we can proceed as in Case \emph{2 $a)$}.  

Combining cases \emph{1} and \emph{2} and Proposition \ref{pr:mixing} we get that for every $\upsilon, \, \nu \in B_\ell(\vb_q)$ there is $\omega$such that $\upsilon \om \nu \in \La(\vb_q)$ and $|\om| = S = \max\set{n_{J} +j , K+1 + j}$. Observe that $|\om|$ does not depend on $\ell$. This gives that $s_\ell(\vb_q) = S$ for every $\ell \in \N$. Then we conclude that $s_{\vb_q} = S$ and that $(\vb_q, \si)$ has the specification property.  
\end{proof}

\begin{corollary}\label{pr:specn}
Let $N \geq 2$. Then, if $q \in \I_N$ then $(\vb_q, \si)$ has specification.
\end{corollary} 
\begin{proof}
Fix $M \in \N$, $N \geq 2$ and let $q \in \I_N$. Then, Proposition \ref{pr:spec3} implies that $q \in \I$. Then $\al(q)$ satisfies that for any $r \geq 2$, $$0^N \prec \al_{(rN)+1}(q) \ldots \al_{(N+1)r}(q) \prec M^N.$$ This implies that $d(\si^n(\al(q)), \overline{\al(q)}) \geq 1/2^N.$ Then the result follows directly from Proposition \ref{pr:spec2}.
\end{proof}

Note that Proposition \ref{pr:spec1} shows the necessity for $q \in \SI$ in order to get the specification property. As we show in the Proposition \ref{pr:spec4} the strong irreducibility of $q$ is not sufficient to get the specification property. In those cases $d(\si^n(\al(q)), \overline{\al(q)})$ is very small for infinitely many $n\in \N$.

\begin{proposition}\label{pr:spec4}
There exists $q \in \SI$ such that $(\vb_q, \si)$ has no specification.
\end{proposition}  
\begin{proof}
We construct now $q \in \SI$ such that for any $K \in \N$ there is $n \in \N$ such that $d(\si^n(\al(q)), \overline{\al(q)}) \leq 1/2^K$. Fix $p_1 \in \operatorname{Per}(\I)$ with quasi-greedy expansion 
$$\al(p_1) = (\al_1(p_1)\ldots\al_{m_1}(p_1))^\f.$$ 
Let $I(p_1)$ the irreducible interval generated by $p_1$ as \cite[Theorem 2]{AlcBakKon2016} and define $q_1$ implicitly by
\[
\al(q_1) = {\al_1(p_1)\ldots\al_{m_1}(p_1)}^+(\overline{\al_1(p_1)\ldots\al_{m_1}(p_1)})^\f.
\]
From Lemma \ref{lem:admissible1} and Lemma \ref{lem:admissible2} there exists infinitely many $n \in \N$ such that $\al_1(q_1) \ldots \al_{n}(q_1)$ is a fundamental word and $(\al_1(q_1) \ldots \al_{n}(q_1))^\f$ is an irreducible sequence. Let 
$$m_2 = \mathop{\max}
\left\{
\begin{array}{clrr}
m \in \set{m_1 + 1, \ldots 2 \cdot m_1} : &\al_1(q_1) \ldots \al_{m}(q_1) \text{ is a fundamental word and }\\
&(\al_1(q_1) \ldots \al_{m}(q_1))^\f \text{ is irreducible}
\end{array}
\right\}.$$
Let $p_2$ be defined implicitly as 
$$\al(p_2) = (\al_1(q_1) \ldots \al_{m_2}(q_1))^\f.$$ 
Set $q_2$ to have quasi-greedy expansion
$$\al(q_2) = {\al_1(p_1)\ldots\al_{m_2}(p_2)}^+(\overline{\al_1(p_2)\ldots\al_{m_2}(p_2)})^\f.$$  
Again, applying \cite[Theorem 2]{AlcBakKon2016} we have that $I(p_2)$ is an entropy plateau. Also note that $p_1 < q_1 < p_2 < q_2$. As before, applying Lemma \ref{lem:admissible1} and Lemma \ref{lem:admissible2} we obtain 
$$m_3 = \mathop{\max}
\left\{
\begin{array}{clrr}
m \in \set{m_2 + 1, \ldots 2 \cdot m_2} : &\al_1(q_2) \ldots \al_{m}(q_2) \text{ is a fundamental word and }\\
&(\al_1(q_2) \ldots \al_{m}(q_2))^\f \text{ is irreducible}
\end{array}
\right\}.$$ 
Then, we can define implicitly $p_3$ to be 
$$\al(p_3) = (\al_1(q_2) \ldots \al_{m_3}(q_2))^\f.$$ 
Assuming that $I(p_n)$ is already defined, we define $p_{n+1}$ to satisfy $$\al(p_{n+1}) = (\al_1(q_n) \ldots \al_{m_{n+1}}(q_n))^\f$$ 
and 
$$m_{n+1} = \mathop{\max}
\left\{
\begin{array}{clrr}
m \in \set{m_n + 1, \ldots 2 \cdot m_n} : &\al_1(q_n) \ldots \al_{m}(q_n) \text{ is a fundamental word and }\\
&(\al_1(q_n) \ldots \al_{m}(q_n))^\f \text{ is irreducible}
\end{array}
\right\}.$$
Then, $p_n < q_n < p_{n+1}$ for every $n \in \N$. Note that, $m_n < m_{n+1}$ for every $n \in \N$. Furthermore, using \ref{constructstrong} we obtain that $p = \mathop{\lim}\limits_{n \to \f} p_n$ exists and $p \in \SI$. As we constructed $p$ we have that 
$$d(\si^{m_n}(\al(p)), \overline{\al(p)}) \leq 1/2^{m_n}.$$ 
Since $m_n < m_{n+1}$ we obtain the desired conclusion.
\end{proof}

\begin{remark}
The procedure described in Subsection \ref{constructstrong} can also be performed to get a base $p \in \SI$ with the specification property. In Proposition \ref{pr:spec4} the sequence $\set{m_n}_{n = 1}^\f$ is not bounded. If a bounded sequence $\set{m_n}_{n = 1}^\f \subset \N$ is considered then the resulting limit point $p$ will satisfy that the subshift $(\vb_p, \si)$ has specification. Here $K = \mathop{\max}\set{m_n}_{n = 1}^\f + 1$ will satisfy the hypothesis of Proposition \ref{pr:spec2}.      
\end{remark}

On the other hand, it is not difficult to show that
$$\operatorname{Per}(\I^*) = 
\set{q \in \I^*: \al(q) \text{ is periodic }}$$ 
will satisfy that there is $N \in \N$ with $d(\si^n(\al(q)), \overline{\al(q)}) \geq 1/2^N$ for every $n \in \N$. However, from \cite[Lemma 3.3]{AlcBakKon2016}, the subshift $(\vb_q, \si)$ is not transitive, thus, it can not have the specification property.

We obtain the following corollary as a direct consequence of \cite[Lemma 6.2]{AlcBakKon2016} Proposition \ref{pr:weakirreduciblesequence}, and Corollary \ref{pr:specn}.

\begin{corollary}\label{cor:c3}
The class $\Ciii$ is uncountable and $(\overline{\B} \cap [q_T, M+1]) \setminus \Ciii$ is an uncountable set.  
\end{corollary}

\section{Synchronised $q$-subshifts}
\label{sec:synchr}

\noindent In this section we characterise the set of $q \in \vl$ such that $(\vb_q, \si)$ is synchronised. Recall that a word $\om \in \La(X)$ for a transitive subshift $(X, \si)$ is \emph{intrinsically synchronising} (colloquially $\om$ is a \emph{magic word}) if whenever $\upsilon\omega$ and $\omega\nu \in \La(X)$ we have $\upsilon\omega\nu \in \La(X).$ We call a transitive subshift $(X, \si)$ to be \emph{synchronised} if there exists an intrinsically synchronising word $\om \in \La(X)$. Following this, let us observe that $\Cv \neq \emptyset$ since for every $q \in \vl \cap (q_{G}, q_T)$, $(\vb_q, \si)$ cannot be a synchronised subshift. Also, notice that $\Civ \subset \I.$

We will show the existence of an intrinsically synchronising word whenever $\al(q)$ is irreducible and the orbit of $\al(q)$ is not dense in $\vb_q$; that is, there is a word $\om \in \La(\vb_q)$ such that $\om$ is not a factor of $\al(q)$. Our intuition is based on Propositions \ref{pr:spec1} and \ref{pr:spec4}.

\begin{lemma}\label{lem:synch1}
If $\al(q)$ is an irreducible sequence and the orbit of $\al(q)$ under $\si$ is not dense in $\vb_q$ then there exists an intrinsically synchronising word $\om \in \La(\vb_q)$.
\end{lemma}
\begin{proof}
We claim that any word $\om \in \La(\vb_q)$ such that $\om$ is neither a factor of $\al(q)$ nor a factor of $\overline{\al(q)}$ is an intrinsically synchronising word. Let $\upsilon = u_1 \ldots u_{\ell}$ and $\nu = v_1 \ldots v_n \in \La(\vb_q)$ such that $\upsilon \om \in \La(\vb_q)$ and $\om \nu \in \La(\vb_q)$. Suppose that $|\om| = m$. Since, $\upsilon \om \in \La(\vb_q)$ we have that 
\begin{equation}\label{eq:ad1}
\overline{\al_1(q)\ldots \al_{\ell + m -i}(q)} \lle u_{i+1} \ldots u_{\ell}w_1 \ldots w_m \lle \al_1(q)\ldots \al_{\ell + m -i}(q)  .
\end{equation}
for every  $i \in \set{0, \ldots \ell -1}.$ Moreover, since $\om \nu \in \La(\vb_q)$ we obtain that
\begin{equation}\label{eq:ad2}
\overline{\al_1(q)\ldots \al_{m + n -j}(q)} \lle w_{j+1} \ldots w_{m}v_1 \ldots v_n \lle \al_1(q)\ldots \al_{m + n -j}(q)  
\end{equation}
for every  $j \in \set{0, \ldots m -1}.$ Suppose that $\upsilon \om \nu \notin \La(\vb_q)$. From \eqref{eq:ad1} and \eqref{eq:ad2} there exists $i \in \set{0, \ldots \ell -1}$ such that either 
$$\al_1(q) \ldots \al_{\ell + m + n - i}(q) \lle  u_{i+1}\ldots u_{\ell}w_1 \ldots w_{m}v_1\ldots v_n$$ 
or 
$$\overline{\al_1(q) \ldots \al_{\ell + m + n - i}(q)} \lge u_{i+1}\ldots u_{\ell}w_1 \ldots w_{m}v_1\ldots v_n.$$ 
Suppose that 
$$\al_1(q) \ldots \al_{\ell + m + n - i}(q) \lle  u_{i+1}\ldots u_{\ell}w_1 \ldots w_{m}v_1\ldots v_n.$$ 
Then, \eqref{eq:ad1} implies the following: either 
\begin{equation}\label{eq:lemma5.1eq1}
u_{i+1}\ldots u_{\ell}w_1 \ldots w_{m} =  \al_1(q) \ldots \al_{\ell + m - i}(q)\quad \text{and} \, v_1 > \al_{\ell + m - i+1}(q)
\end{equation}
 or there is $j' \in \set{2, \ldots n}$ such that 
\begin{equation}\label{eq:lemma5.1eq2}
u_{i+1}\ldots u_{\ell}w_1 \ldots w_{m}v_1\ldots v_{j'-1} = \al_1(q) \ldots \al_{\ell + m + j'-1 - i}(q) \quad \text{and} \, v_j > v_{\ell + m + j' -i}.
\end{equation} 
Both \eqref{eq:lemma5.1eq1} and \eqref{eq:lemma5.1eq2} contradict that $\omega$ is not a factor of $\al(q)$. If $$\overline{\al_1(q) \ldots \al_{\ell + m + n - i}(q)} \lge u_{i+1}\ldots u_{\ell}w_1 \ldots w_{m}v_1\ldots v_n$$ the proof follows from a similar argument. 
\end{proof}

The following proposition gives a sufficient condition on $\al(q)$ to guarantee that there exists $q \in \I$ such that $(\vb_q, \si)$ is transitive and non-synchronised.

\begin{proposition}\label{pr:synch2}
If $q \in \I$ and $\al(q)$ has dense orbit under $\si$ then no word $\om \in \La(\vb_q)$ is intrinsically synchronising.
\end{proposition}
\begin{proof}
Let $\om \in \La(\vb_q)$ and let us assume that $|\om| = n$. Since the orbit under $\si$ of $\al(q)$ is dense in $\vb_q$, it is clear that the orbit of $\overline{\al(q)}$ is also dense in $\vb_q$. Recall that the \emph{follower set} of $\om$ is given by
$$F_{\vb_q}(\om) = \set{\nu \in \La(\vb_q): \om \nu \in \La(\vb_q)}.$$ 
Fix $m \in \N$ and let ${F^m_{\vb_q}}(\om) = \set{\nu \in F_{\vb_q}(\om) : |\nu| = m}$. Since $\al(q)$ and $\overline{\al(q)}$ have dense orbits in $\vb_q$ then for every $\nu \in {F^m_{\vb_q}}(\om)$ there exist $k, k' \in \N$ such that $\om\nu$ is a prefix of $\si^k(\al(q))$ and $\om\nu$ is a prefix of $\si^{k'}(\overline{\al(q)})$. Fix $\gamma, \nu  \in {F^m_{\vb_q}}(\om)$ with $\gamma \prec \nu$. Then, there exist $k$ and $k' \in \N$ such that $\si^{k}(\al(q)) = \om \gamma$ and $\si^{k'}(\overline{\al(q)}) = \om\nu$. Clearly $\om \gamma \prec \om \nu$. Let $\upsilon = \al_1(q) \ldots \al_{k}(\al(q))$. Then $\upsilon \om \in \La(\vb_q)$. However $\upsilon \om \nu \notin \La(\vb_q)$ since $\upsilon \om \nu \succ \al_1(q) \ldots \al_{k+n+m}$. Thus, $\om$ is not a synchronising word.   
\end{proof}

As a direct consequence of \eqref{eq:sequenceofshifts} and Corollary \ref{cor:c3}, we obtain the following corollary.

\begin{corollary}\label{cor:c4}
The class $\Civ$ is uncountable.
\end{corollary}

\subsection{Existence non synchronised and transitive symmetric $q$-subshifts.}
\label{constructdense}

We will show now that there are bases $q \in \I$ with $\al(q)$ dense in $\vb_q$. We will perform a construction inspired by the one given by Schmeling in \cite[Proof of Theorem B]{Sch1997}. Unfortunately, the presented construction is algorithmically complicated. Nonetheless, it is not our objective to give an optimal construction for such bases.

Firstly, fix $q \in \operatorname{Per}(\I)$. So, $\al(q) = (\al_1(q) \ldots \al_m(q))^\f$ where $m$ is the period of $\al(q)$. Consider $B_m(\vb_q)$ ordered in a decreasing way with respect to the lexicographical order. Note that the largest element of $B_m(\vb_q)$ is $\al_1(q) \ldots \al_m(q)$ and the smallest is $\overline{\al_1(q) \ldots \al_m(q)}$. Now we recall \cite[Lemma 3.9]{AlcBakKon2016}.

\begin{lemma}\label{lem:prefix}
Let $q\in [q_T, M+1] \cap \vl$ and $m\in\N$. Then for any $\nu\in\La(\vb_q)$ there exists $\eta\in\La(\vb_q)$ such that $k^m$ is a prefix of $\eta\nu\in\La(\vb_q)$ if $M=2k$, or $((k+1)k)^m$ is a prefix of $\eta \nu \in\La(\vb_q)$ if $M=2k+1$.
\end{lemma}   

Then, for every $\om \in B_n(\vb_q)$ there is a word $\eta_{\omega} \in \La(\vb_q)$ such that $(k+1)k$ is a prefix of $\eta_{\om}\om \in \La(\vb_q)$ if $M = 2k+1$, or $k$ is a prefix of $\eta_{\om}\om \in \La(\vb_q)$ if $M = 2k$. On the other hand, Lemma \ref{lem:suffix} assures us that for every $j > |\eta_{\om}\om|$, there exists a word $\gamma_{\eta_{\om}\om} \in \La(\vb_q)$ such that $\eta_{\om}\om\gamma_{\eta_{\om}\om} \in \La(\vb_q)$ and the word $\al_1(q) \ldots \al_m(q)$ is a suffix of $\eta_{\om}\om\gamma_{\eta_{\om}\om}$.  For each $i \in \set{2, \ldots, \#B_m(\vb_q)-1}$ and for each word $\om_i \in B_m(\vb_q)$ we call the word
\begin{equation}\label{def:extendedword}
\omega'_i = \eta_{\om_i}\om_i\gamma_{\eta_{\om_i}\om_i}
\end{equation} 
\emph{an extended word of $\om_i$}. For $\om_1 = \al_1(q)\ldots \al_m(q)$ we set $\om'_1 = (\al_1(q)\ldots \al_m(q))^t$ for a fixed $t \in \N$ with $t \geq 2$,\footnote{This technical condition will help us to prove that the word generated in Lemma \ref{lem:densegenerator} will parameterise a transitive symmetric subshift.} and for $\om_{\#B_m(\vb_q)} = \overline{\al_1(q) \ldots \al_m(q)}$ we set $\om'_{\#B_m(\vb_q)} = \om_{\#B_m(\vb_q)}$.

From \cite[Theorem 1]{AlcBakKon2016} there exists $\set{\upsilon_i}_{i = 1}^{\#B_m(\vb_q) -1} \subset \La(\vb_q)$ such that 
\begin{equation}\label{eq:denseword}
\delta(q) = \om'_1\upsilon_1\om'_2\upsilon_2 \ldots \om'_{\#B_m(\vb_q)-1}\upsilon_{\#B_m(\vb_q)-1} \om'_{\#B_m(\vb_q)} \in \La(\vb_q).
\end{equation}  
Clearly, the word $\delta(q)$ defined in \ref{eq:denseword} satisfies that $\nu$ is a factor of $\delta(q)$ for every $\nu \in \mathop{\bigcup}\limits_{k = 0}^m B_k(\vb_q)$.

We recall now the notion of \emph{primitive word} introduced in \cite[Definition 3.10]{AlcBakKon2016}. Given a finite word $\om = w_1 \ldots w_n$ we say that $\om$ is \emph{primitive} if 
\begin{equation}\label{eq:primitive}
\overline{w_1 \ldots w_{n-i}} \prec w_{i+1} \ldots w_n \lle w_1 \ldots w_{n-i}\quad \text{ for every } i \in \set{0, \ldots m-1}. 
\end{equation} 
Note that  
\eqref{eq:primitive} is similar to \eqref{eq:admissible1}. Also, it is clear that 
\begin{equation}\label{eq:distance}
d(\delta(q), \al_1(q) \ldots \al_{|\delta(q)|}) \leq 1/2^{m \cdot t}.
\end{equation} 
Here, we are inducing the distance $d$ for $\Sig$ as a distance in $B_{|\delta(q)|}(\vb_q)$. 

The main idea of our construction is to show that for any $q \in \operatorname{Per}(\I)$ it is possible to find a set $\set{\upsilon_i}_{i = 1}^{\#B_m(\vb_q)-1} \subset \La(\vb_q)$ satisfying that $\delta(q)$ is a prefix of a fundamental word $\theta(q)$. For this purpose, we need to recall the notion of \emph{reflection recurrence word} introduced in \cite[Definition 3.11]{AlcBakKon2016} and some related results. 

Given a primitive word $\om = w_1 \ldots w_n$, the \emph{reflection recurrence word of $\om$} is the truncated word
$$\mathfrak{R}(\om) =  w_1 \ldots w_s$$ 
where 
$$s = \mathop{\min}\set{s \in \set{0, \ldots, n-1}: w_{s+1}\ldots w_n^- = \overline{w_1 \ldots w_{n-s}}}.$$ 
In the case that $s = 0$, we have that $\mathfrak{R}(\om) = \epsilon$. We summarise the results proven in \cite[Lemmas 3.13, 3.14. 3.15 and 3.16]{AlcBakKon2016} in the following lemma.

\begin{lemma}\label{lem:reflectionreccurrence}
Suppose that $\om$ is a primitive word with $|\om| = m$.
\begin{enumerate}[$i)$]
\item If  $m \geq 2$ then 
$$m/2 \leq |\mathfrak{R}(\om)| \leq m;$$
\item $\mathfrak{R}(\om)$ is primitive;
\item For $n \in \N$ set $\mathfrak{R}^n(\om) = \mathfrak{R}(R^{n-1}(\om))$ and $\mathfrak{R}^0(\om) = \mathfrak{R}(\om)$. If $m \geq 2$ then there exists $j \in \set{0, \ldots, m}$ such that either 
$$\mathfrak{R}^{j+1}(\om) = \mathfrak{R}^{j}(\om)$$ 
with $|\mathfrak{R}^j(\om)| \leq 2$ or $\mathfrak{R}^j(\om) = w_1\overline{w_1}^+$.  
\item Let $q \in \I$. There exists infinitely many $m \in \N$ such that $\al_1(q) \ldots \al_m(q)$ is a primitive word and for each of such $m \in \N$ there exists $N = N(m)$ such that 
\begin{align}\label{eq:primitive1} 
\overline{\al(q^-_N)} &\prec \si^n(\al_1(q) \ldots \al_m(q)(\overline{\al_1(q) \ldots \al_m(q)}^+)^\f) \prec \al(q^-_N);\\
\label{eq:primitive1prime}
\overline{\al(q^-_N)} &\prec \si^n(\overline{\al_1(q) \ldots \al_m(q)}(\al_1(q) \ldots \al_m(q)^-)^\f) \prec \al(q^-_N);
\end{align}
and for every $r \in \N$ 
\begin{align}
\label{eq:primitive2} 
\overline{\al(q^-_N)} &\prec \si^n({\al_1(q) \ldots \al_m(q)}^-(\mathfrak{R}(\al_1(q) \ldots \al_m(q))^-)^\f) \prec \al(q^-_N);\\
\label{eq:primitive2prime}
\overline{\al(q^-_N)} &\prec \si^n(\overline{\al_1(q) \ldots \al_m(q)}^+(\mathfrak{R}(\al_1(q) \ldots \al_m(q))^+)^\f) \prec \al(q^-_N).
\end{align}
\end{enumerate}
\end{lemma}  

We now show the desired properties for $\delta(q)$. We want to remark that the proof of the following lemma is strongly based on the argument used to prove \cite[Propositon 3.17]{AlcBakKon2016}.

\begin{lemma}\label{lem:primitive}
Let $q \in \operatorname{Per}(\I)$. Then there exists $\set{\upsilon_i}_{i = 1}^{\#B_m(\vb_q) -1} \subset \La(\vb_q)$ such that $\delta(q)$ is a prefix of a fundamental word $\theta(q)$.
\end{lemma} 
\begin{proof}
Let $\al(q) = (\al_1(q)\ldots \al_m(q))^\f$ and consider $B_m(\vb_q)$ ordered in a decreasing way with respect to the lexicographical order. For each $\om_i \in B_m(\vb_q)$ let $\om'_i$ be a extended word of $\om_i$ given in \ref{def:extendedword}. From Lemma \ref{lem:reflectionreccurrence} there is a word $\ga_1$ such that 
\begin{equation}\label{eq:proofprim1}
\om_1\ga_1 \in \La(\vb_q) \quad \text{ and } \quad \om_1\ga_1 \quad \text{ is primitive.} 
\end{equation}
We can consider $\ga_1$ to have minimal length and satisfy \eqref{eq:proofprim1}. From \cite[Proposition 3.17]{AlcBakKon2016} there exists $\nu_1 \in \La(\vb_q)$ such that $\om'_1\ga_1\nu_1\om'_2 \in \La(\vb_q)$ and 
$$\nu_1 = (\overline{\mathfrak{R}(\om'_1\ga)}^+)^N \ldots \overline{\mathfrak{R}^{\ell}(\om'_1\ga_1)}^+)^N,$$ 
where $\ell \in \set{1, \ldots, |\om'_1\ga|}$ and $N$ is given by \eqref{eq:primitive2prime} of Lemma \ref{lem:reflectionreccurrence}. Let us set $$\om'_1\ga\nu_1\om'_2 = b^1_1 \ldots b^1_K$$ with $K_1 = |\om'_1\ga\nu_1\om'_2|$. Since $\om_1 \succ \om_2$ we have that for any $j \in \set{1, \ldots, K_1},$ $$\overline{b^1_1 \ldots b_{K_1-j}} \lle \si^{j}(b^1_1 \ldots b_{K^1}) \lle b^1_1 \ldots b^1_{K_1-j}.$$  Since $\om'_2$ is the suffix of $b_1 \ldots b_K$ we have that $\al_1(q)\ldots\al_m(q)$ is the suffix of length $m$ of $b_1 \ldots b_K$. Then, from Lemma \ref{lem:reflectionreccurrence} there is $\ga_2$ such that $\al_1(q)\ldots\al_m(q)\ga_2 \in \La(\vb_q)$ is primitive. We can consider again $\ga_2$ to have minimal lenght. Then \cite[Proposition 3.17]{AlcBakKon2016} implies that there is $\nu_2 \in \La(\vb_q)$ such that $\om'_1\ga_1\nu_1\om'_2\ga_2\nu_2\om'_3 \in \La(\vb_q)$ and
$$\nu_2 = (\overline{\mathfrak{R}(\al_1(q)\ldots\al_m(q)\ga_2)}^+)^N \ldots \overline{\mathfrak{R}^{\ell}(\al_1(q)\ldots\al_m(q)\ga_2)}^+)^N,$$
with $\ell \in \set{1, \ldots, m+|\ga_2|}$ and $N \in \N$. Set $K_2 = |\om'_1\ga_1\nu_1\om'_2\ga_2\nu_2\om'_3|$ and
$$\om'_1\ga_1\nu_1\om'_2\ga_2\nu_2\om'_3 = b^2_1 \ldots b^2_{K_1}.$$ 
Similarly, since $\om_2 \succ \om_3$ then for any $j \in \set{1, \ldots, K_2}$ we obtain that 
$$\overline{b^2_1 \ldots b_{K_2-j}} \lle \si^{j}(b^2_1 \ldots b_{K^2}) \lle b^2_1 \ldots b^2_{K_2-j}.$$ 
Iterating this procedure we obtain that the set $\set{\upsilon_i}_{i=1}^{\#B_m(\vb_q)}$ with $\upsilon_i = \ga_i\nu_i$ for every $i \in \set{1, \ldots \#B_m(\vb_q)-1}$ satisfies that 
$$\delta(q) = b_1\ldots b_K = \om'_1\upsilon_1\om'_2\ldots \upsilon_{\#B_m(\vb_q)-1}\om'_{\#B_m(\vb_q)} \in \La(\vb_q)$$ 
and 
$$ \overline{b_1\ldots b_{K-j}} \lle \si^j(\delta(q)) \lle b_1\ldots b_{K-j}$$ 
for every $j \in \set{0, \ldots, K-1}$. Observe that for $j = K-m$ we have that $\si^j(\delta(q)) =  \overline{\al_1(q)\ldots \al_m(q)}$. Consider $\theta(q) = \delta(q)\overline{\al_1(q)}$. Note that $\si^j(\theta(q)) = \overline{\al_1(q)\ldots\al_m(q)\al_1(q)}$. Since $\al_1(q)\ldots\al_m(q)$ is a fundamental word and satisfies
\[
\overline{\al_1(q)\ldots\al_m(q)\al_1(q)} \prec \al_1(q)\ldots\al_m(q)\al_1(q),
\]
we have that for every $j \in \set{1, \ldots m}$,
\[
\overline{\al_1(q)\ldots\al_{m+1-i}(q)} \lle \overline{\al_{i+1}(q) \ldots \al_m(q)\al_1(q)} \prec \al_1(q)\ldots\al_{m+1-i}(q).
\] 
This implies that $\theta(q)$ is a fundamental word.        
\end{proof}

As a consequence of Lemma \ref{lem:primitive} we obtain the following result.

\begin{lemma}\label{lem:densegenerator}
Let $q \in \operatorname{Per}(\I)$ with $\al(q) = (\al_1(q) \ldots \al_m(q))^\f$ and let $\theta(q) = \theta(q, t)$ given by Lemma \ref{lem:primitive}. Then:
\begin{enumerate}[$i)$]
\item $\theta(q)$ is a fundamental word;
\item For any $\nu \in \mathop{\bigcup}\limits_{k = 0}^m B_k(\vb_q)$, $\nu$ is a factor of $\theta(q)$.
\item $(\theta(q))^\f \prec \al(q)$
\item $(\theta(q))^\f$ is irreducible.
\end{enumerate}
\end{lemma}
\begin{proof}
Note that $i)$ and $ii)$ are direct consequences of Lemma \ref{lem:primitive}. Moreover, since $\om_1 \succ \om_2$ we obtain $iii)$ directly from the definition of $\theta(q)$. It remains to show that $(\theta(q))^\f$ is irreducible. Let $j \in \N$ such that $(\theta_1(q) \ldots \theta_j(q)^-)^\f \in \vb$. Then if $j \in \set{1, \ldots, t \cdot m}$ then 
\begin{equation}\label{eq:laultima}
\theta_1(q) \ldots \theta_j(q)(\overline{\theta_1(q) \ldots \theta_j(q)}^+)^\f \prec (\theta(q))^\f
\end{equation} 
holds from the irreducibility of $\al(q)$. On the other hand for $t\cdot m \geq j$ note that since $\overline{\al_1(q)\ldots \al_m(q)} \prec \om'_i$ for every $i \in \set{1, \ldots, \#B_m(\al(q))-1}$ and $\overline{\al_1(q)\ldots \al_m(q)}^t$ is a factor of $\overline{\al_1(q) \ldots \al_j(q)}$ we have that \eqref{eq:laultima} holds. Thus $\theta(q)$ is irreducible.  
\end{proof}

\begin{proposition}\label{pr:existnonsynchsubshift}
There exists $q \in \I$ such that $\set{\sigma^n(\al(q))}_{n=0}^\f$ is a dense subset of $\vb_q$. 
\end{proposition}
\begin{proof}
Let $q_1 \in \operatorname{Per}(\I)$ with quasi-greedy expansion $\al(q_1) =  (\al_1(q_1) \ldots \al_{m_1}(q_1))^\f$. Fix $t_1 \in \N$ with $t_1 \geq 2$. From Lemma \ref{lem:primitive} and Lemma \ref{lem:densegenerator} there is a fundamental word $\theta(q_1)$ such that: 
\begin{enumerate}[$i)$]
\item \label{1} $(\theta(q_1))^\f$  is irreducible;  
\item \label{2} $(\theta(q_1))^\f \prec \al(q_1)$;   and    
\item \label{3} $\nu$  is a factor of  $\theta(q_1)$  for any $\nu \in \mathop{\bigcup}\limits_{k = 0}^{m_1}B_k(\vb_{q_1}).$ 
\end{enumerate} 
Since $(\theta(q_1))^\f$ is irreducible then $\al(q_T) \prec (\theta(q_1))^\f$. Set $q_2$ be defined explicitly such that
$$\al(q_2) = (\theta(q_1))^\f = (\al_1(q_2) \ldots \al_{m_2}(q_2))^\f.$$
Clearly $m_1 < m_2$. Also, \eqref{eq:distance} implies that 
$$d(\al(q_1), \al(q_2)) < 1/2^{m_1 \cdot t_1}.$$
Moreover, from Lemma \ref{lem:quasigreedyexpansion} we have that $q_1 \geq q_2$. This combined with \cite[1.5.10]{LinMar1995} imply that 
$$B_{m_1}(\vb_{q_1}) = B_{m_1}(\vb_{q_2}).$$ 
Then, \emph{\ref{3})} implies that $\nu$ is a factor of $\al(q_2)$ for every $\nu \in \mathop{\bigcup}\limits_{k = 0}^{m_1}B_k(\vb_{q_2}).$

Consider $t_n \in \N$ with $t_n \geq 2$ for every $n \geq 2$. Suppose that $q_n$ is already defined. We define $q_{n+1}$ implicitly as $\al(q_{n+1}) = (\theta(q_n))^\f = (\theta(q_n, t_n))^\f$.

Note that from \emph{\ref{1})} we have that $\al(q_n)$ is irreducible, so $q_n \in \I$ for every $n \in \N$. Also, Lemma \ref{lem:quasigreedyexpansion} and $\ref{2}$ imply that $\set{q_n}_{n = 1}^\f$ is a decreasing sequence. Furthermore, $\set{q_n}$ is bounded from below by $q_T$. Thus $q_n \searrow q \in \overline{\ul} \cap [q_T, M+1]$ as $n \to \f$. This implies  
\[
\vb_q \subset \mathop{\bigcap}\limits_{n = 1}^\f \vb_{q_n}.
\]
Let us consider the quasi-greedy expansion of $q$, $\al(q)$. If $\x = (x_i) \in \mathop{\bigcap}\limits_{n = 1}^\f \vb_{q_n}$ then 
\[
\overline{\al(q_n)} \lle \si^m((x_i)) \lle \al(q_n) \quad \text{for every} \, n, m \in \N.
\] 
Therefore, $\overline{\al(q_n)} \prec \si^m((x_i)) \prec \al(q_n)$ for every $n, m \in \N$, which gives that $\overline{\al(q)} \lle \si^m((x_i)) \lle \al(q)$, i.e. $x \in \vb_q$. Thus, 
\begin{equation}\label{eq:intersection}
\vb_q = \mathop{\bigcap}\limits_{n = 1}^\f \vb_{q_n}.  
\end{equation}
Since, $m_n < m_{n+1}$, it is clear that $\al(q)$ is not a periodic sequence. Note that \emph{\ref{3})} implies that for every $n \in\N$ and for every $\nu \in \mathop{\bigcup}\limits_{k = 0}^{m_n}B_k(\vb_{q_n})$, $\nu$ is a factor of $\al(q_{n+1})$.  Also, observe that for every $n \in \N$, $\al_1(q_n) \ldots \al_{m_n}(q_n)$ is a prefix of $\al(q)$. Then, for any 
$$\nu \in \mathop{\bigcup}\limits_{n =1}^\f \mathop{\bigcup}\limits_{k = 0}^{m_n}B_k(\vb_{q_n}),$$
$\nu$ is a factor of $\al(q)$. Fix $\omega = w_1 \ldots w_k \in \La(\vb_q)$. Then 
$$\overline{\al_1(q) \ldots \al_{k-i}(q)} \lle w_{i+1} \ldots w_k \lle \al_1(q) \ldots \al_{k-i}(q) \quad \text{for every }\, i \in \set{0, \ldots, k-1}.$$ 
Then, from \eqref{eq:intersection} and since $q_n \searrow q$, \cite[1.5.10]{LinMar1995} implies that there is $N \in \N$ such that for every $n \geq N$, $\om \in B_{k}(\vb_{q_n}) = B_{k}(\vb_{q_N})$ and $B_k(\vb_{q_n}) = B_k(\vb_q)$. Thus, $\om$ is a factor of $\al(q_n)$ for every $n \geq N$ which gives that $\om$ is a factor of $\al(q)$. Then, since each word in $\La(\vb_q)$ has to appear in $\al(q)$ infinitely often we obtain that $\set{\si^n(\al(q))}_{n=0}^\f$ is a dense subset of $\vb_q$ and that $\vb_q$ is a transitive subshift. Therefore, \cite[Theorem 1]{AlcBakKon2016} gives that $q \in \I$.     
\end{proof}

\begin{corollary}\label{pr:notsyncisdense}
For any $q \in \overline{\B} \cap [q_T, M+1]$ there is $p$ such that $\set{\sigma^n(\al(p))}_{n=0}^\f$ is dense in $\vb_p$ and $p$ is arbitrarily close to $q$, i.e. $\set{p \in \I: \al(p) \text{ is dense in } \vb_p}$ is dense in $\overline{\B} \cap [q_T, M+1]$.
\end{corollary}
\begin{proof}
Let $q \in \overline{\B} \cap [q_T, M+1]$ with quasi-greedy expansion $\al(q)$.  We note here that if $q \in I(p')$ for some $p' \in \operatorname{Per}(\I)$ then, from \cite[Proposition 4.11]{AlcBakKon2016} Lemma \ref{lem:infinitelymanyirreducibles} we obtain that $q = p'$, or $q$ satisfies 
$$\al(q) = \al_1(p') \ldots \al_{m(p')}(p')^+(\overline{\al_1(p') \ldots \al_{m(p')}(p')})^\f
,$$ 
where $m(p')$ is the period of $\al(p')$. Then, we have to consider three cases:

\noindent \emph{\textbf{Case 1}}: Suppose that $q \in I(p')$ and $q = p'$. Let $\al(q) = (\al_1(q) \ldots \al_{m}(q))^\f$ the quasi-greedy expansion of $q$. Then, since $\operatorname{Per}(I)$ is dense in $\overline{\B} \cap [q_T, M+1]$ then, for every $N \in \N$ there exists $p_N \in \operatorname{Per}(\I)$, $\al(p_N) =  (\al_{1}(p_N) \ldots \al_{m_N}(p_N))^\f$ such that $d(\al(p_N), \al(q)) \leq 1/2^{N+1}$. In particular, for any $\varepsilon > 0$ there are $N, M \in \N$ such that 
$$d(\al(p_N), \al(q)) \leq 1/2^{(M \cdot m)+1} < 1/2^N < \varepsilon/2.$$ 
So, fix $t_1 \in \N$ with $t_1 \geq 2$ and consider $\theta(p_N)$ given by Lemma \ref{lem:densegenerator} and set $p_{N_1}$ defined implicitly by $\al(p_{N_1}) = (\theta(p_N))^\f$. Since $t_1 \geq 2$, we get 
$$d(\al(p_{N_1}, \al(p_N)) \leq 1/2^{2 \cdot m_N}.$$
Then, 
$$d(\al(p_{N_1}), q) \leq 1/2^{2 \cdot m_N} + 1/2^{(M \cdot m)+1} < 2/2^{(M \cdot m)+1} \leq \varepsilon. 
$$ 
Consider $\set{\theta(p_{N_i})}_{i = 1}^\f$, the sequence generated in the proof of Proposition \ref{pr:existnonsynchsubshift} and the associated sequence $\set{\al(p_{N_i})}_{i=1}^\f$ with limit $p$. Note that for every $i \geq 2$, 
$$\al_1(p_{N_i}) \ldots \al_{t_1 \cdot m_{N_1}}(p_{N_i}) = \al_1(p_{N_1}) \ldots \al_{t_1 \cdot m_{N_1}}(p_{N_1}),$$ 
where $m_{N_1}$ is the period of $\al(p_{N_1})$. Then, $d(\al(p_{N_i}), \al(q)) < \varepsilon$ for every $i \in \N$, thus $d(\al(p), \al(q)) < \varepsilon.$ As a consequence of Lemma \ref{lem:quasigreedyexpansion} we obtain the result for this case.

\noindent \emph{\textbf{Case 2}}: Suppose that $q \in I(p')$ and $\al(q) = \al_1(p') \ldots \al_{m(p')}(p')^+(\overline{\al_1(p') \ldots \al_{m(p')}(p')})^\f$. Let $\varepsilon > 0$. From Lemmas \ref{lem:admissible1} and \ref{lem:admissible2} there exists $N\in \N$ such that for every $n \geq N$, 
\begin{enumerate}[$i)$]
\item \label{prop1} $\al_1(q) \ldots \al_{m_n}(q)$ is a fundamental word;
\item \label{prop2} $(\al_1(q) \ldots \al_{m_n}(q))^\f$ is an irreducible sequence; and 
\item \label{prop3} $d((\al_1(q) \ldots \al_{m_n}(q))^\f, \al(q)) < 1/2^{N \cdot m_n}< \varepsilon.$ 
\end{enumerate}
Fix $n \in \N$ such that \emph{\ref{prop1}), \ref{prop2}) \ref{prop3})} holds for for $\varepsilon$ and let $q'$ with $\al(q') = (\al_1(q) \ldots \al_{m_n}(q))^\f$. Fix $t_1 \in \N$ with $t_1 \geq 2$. Let $\theta(p_1)$ given by Lemma \ref{lem:densegenerator} and set $p_1$ defined implicitly by $\al(p_1) = (\theta(q'))^\f$. Fix a sequence $\set{t_n}_{n = 2}^\f \subset \N$ with $t_n \geq 2$ for every $n \geq 2$. Then, for every $n \geq 2$ consider $\set{\theta(p_n)}_{i = 2}^\f$ the sequence generated in the proof of Proposition \ref{pr:existnonsynchsubshift} and the associated sequence $\set{\al(p_{n})}_{i=1}^\f$. From a similar argument as in Case $1$, the quasi-greedy expansion of $p$, $\al(p)$, where $p$ is the limit of the sequence $\set{p_n}_{n = 1}^\f$ satisfies the desired conclusion. 

\noindent \emph{\textbf{Case 3}}: Suppose that $q \in (\overline{\B} \cap [q_T, M+1]) \setminus \mathop{\bigcup}\limits_{p' \in \operatorname{Per}(\I)}I(p')$. In such case, note that for any $\varepsilon > 0$, there are $p^+, p^- \in \operatorname{Per}(\I)$ such that 
$$d(\al(p^-), \al(q)) < \varepsilon/2 \quad \text{ and } \quad \al(p^-) \prec \al(q);$$ 
and 
$$d(\al(p^+), \al(q)) < \varepsilon/2 \quad \text{ and }\quad \al(p^+) \succ \al(q).$$ 
Then, applying the construction exposed on Case $1$ to $\al(p^-) = (\al_1(p^-)\ldots \al_{m_{p^-}}(p^-))^\f$ and the construction exposed in Case $2$ to $\al(p^+) = (\al_1(p^+)\ldots \al_{m_{p^+}}(p^+))^\f$ we can construct the desired base.
\end{proof}

As in the previous constructions \ref{constructstrong} and \ref{constructweak}, the construction of a base $q$ with dense quasi-greedy expansion $\al(q)$ in $\vb_q$ depends on a sequence $\set{t_n}_{n = 1}^\f$. Then, the following corollary holds.

\begin{corollary}\label{cor:c5}
The subclass $\Cv \cap \I$ is uncountable.
\end{corollary}

\section{Hausdorff dimension of the classes $\Ciii, \Civ, \Cv$}
\label{sec:hd}

\noindent We will investigate the Hausdorff dimension of the classes $\Ciii, \Civ, \Cv$. Let us summarise the results concerning of this classes so far.

\noindent (A) $\Ciii = \set{q \in \I: q \in \SI \text{ and there is } K \in \N \text{ with }d(\si^n(\al(q)), \overline{\al(q)}) \geq 1/2^K \text{ for } n \in \N}$; 

\noindent (B) $\Civ = \set{q \in \I: \si^n(\al(q)) \text{ is not dense in } \vb_q}$;

\noindent (C) $\Cv = \set{q \in \I: \si^n(\al(q)) \text{ is dense in } \vb_q}$;

Observe that (A) is a consequence of Proposition \ref{pr:spec1}, Proposition \ref{pr:spec2} and Proposition \ref{pr:spec4}. Also, (B) and (C) are consequences of Lemma \ref{lem:synch1} and Proposition \ref{pr:synch2}. Finally, Corollary \ref{cor:c3}, Corollary \ref{cor:c4}, Proposition \ref{pr:notsyncisdense} and Corollary \ref{cor:c5} imply $\Ciii, \Civ,$ and $\Cv$ are uncountable and dense subsets of $\overline{\B}\cap[q_T, M+1]$. 

Let us establish the necessary results to complete the proof of Theorem \ref{th:symshift1}. From Theorem \ref{th:mainrsd} we have that $\dim_{\operatorname{H}}(\B) = 1$. Furthermore, from \cite[Lemma 4.10]{AlcBakKon2016} we have 
$$\overline{\B} \subset \overline{\ul} \subset \vl.$$ 
Then, combining \cite[Theorem 3]{AlcBakKon2016}, \cite[Lemma 2.6]{KalKonLiLu2017} it holds that
$$\dim_{\operatorname{H}}(\overline{\B}) = \dim_{\operatorname{H}}(\overline{\ul}) = \dim_{\operatorname{H}}(\vl) = 1.$$ 

Now, from \cite[Theorem 2]{KalKonLiLu2017} we know that for any $q \in \overline{\B}$ we have that $\dim_{\operatorname{H}}^{\operatorname{loc}}(\overline{\B}_q) = \dim_{\operatorname{H}}(\V_q)$. 

We recall now \cite[Lemma 6.4]{AlcBakKon2016}.

\begin{lemma}\label{lem:64}
Let $N\ge 2$. Then 
$$\dim_{\operatorname{H}}(\I_N)\ge \log((M+1)^N-2)/N\log(M+1).$$
\end{lemma}

Then, from Lemma \ref{lem:64} we obtain directly the following:

\begin{proposition}\label{pr:hdc3} 
$\dim_{\operatorname{H}}(\Ciii) = 1.$ Moreover, $\dim_{\operatorname{H}}(\Civ) = 1.$
\end{proposition}
\begin{proof}
As in the proof of \cite[Theorem 3]{AlcBakKon2016}, we note that $\I_N \subset \Ciii$ for every $N \geq 2$. From Lemma \ref{lem:64} the statement holds by letting $N \to \f$. Finally, \eqref{eq:sequenceofshifts} gives us $\dim_{\operatorname{H}}(\Civ) = 1.$
\end{proof}

Now, as a consequence of Lemma \ref{th:hdderong}, Proposition \ref{pr:continuityinB} and Proposition \ref{pr:denseweakirr} we have $$\dim_{\operatorname{H}}(\WI) =\dim_{\operatorname{H}}(([q_T, M+1] \cap \vl)\setminus \Ciii) = 1,$$ thus
$$\dim_{\operatorname{H}}(\vl\setminus \Ciii) = 1 .$$ 

Finally, Lemma \ref{th:hdderong}, Proposition \ref{pr:continuityinB} and Corollary \ref{pr:notsyncisdense} imply the following statement.

\begin{proposition}\label{pr:dimHC5}
$$\dim_{\operatorname{H}}(\Cv) = 1.$$
\end{proposition}

\begin{proof}[\textbf{\textit{Proof of Theorem \ref{th:symshift1}}}]
Theorem \ref{th:symshift1} follows from Propositions \ref{pr:spec1}, \ref{pr:spec2}, \ref{pr:spec4}. \ref{pr:hdc3}, \ref{pr:synch2}, \ref{pr:dimHC5} and Lemma \ref{lem:synch1}.
 \end{proof}

\section{Final comments and open questions}
\label{sec:open}

\noindent We notice that Definition \ref{def:strongweaksequences} can be generalised in the following way. We say that a sequence $\al = (\al_i) \in \vb$ is \emph{strongly $*$-irreducible} if $\al$ is $*$-irreducible and there exists $N \in \N$ such that for every $n \geq N$ with 
$(\al_1 \ldots \al_n^-)^\f \in \vb$, then $(\al_1 \ldots \al_n^-)^\f$ is $*$-irreducible. Similarly, a sequence $(\al_i)^\f \in \vb$ is \emph{weakly $*$-irreducible} if $\al$ is $*$-irreducible and there exist infinitely many $n \in \N$ such that $(\al_1 \ldots \al_n^-)^\f \in \vb$ and $(\al_1 \ldots \al_n^-)^\f$ is not $*$-irreducible. We can also define \emph{strongly irreducible numbers} and \emph{weakly irreducible numbers} as in Definition \ref{def:strongweaknumber}. 

Using the constructions performed in Subsections \ref{constructstrong} and \ref{constructweak} with small modifications it is not difficult to see that there exist uncountably many strongly $*$-irreducible sequences and uncountably many weakly $*$-irreducible sequences. Also, it is not hard to check that strongly irreducible numbers and weakly irreducible numbers are dense in $(q_{KL} ,q_T) \cap \vl$ where $q_{KL}$ is defined in \eqref{eq:23}. Then, it is natural to investigate the Hausdorff dimension of the set of strongly $*$-irreducible numbers and $*$-weakly irreducible numbers.

\begin{question}\label{q:hd}
What is the Hausdorff dimension of strongly irreducible and weakly irreducible numbers? 
\end{question}

We would like to make some comments about the symbolic dynamics of $(\vb_q,\si)$ when $q \in (q_{KL}, q_T)$. In \cite[Proposition 5.10]{AlcBakKon2016} it was shown that if $q$ is a $*$-irreducible base $q$ such that $\al(q)$ is periodic then for every $p \in I^*(q)$, $\vb_p$ contains a unique subshift of finite type $(X,\si)$ such that $h_{\operatorname{top}}(X) = h_{\operatorname{top}}(\vb_q) = h_{\operatorname{top}}(\vb_p)$. To the best of the knowledge of the author it is not known that $(X,\si)$ is a mixing subshift. Also, if $q \in \vl \cap [q_{KL}, q_T)$ it is known that $(\vb_q, \si)$ is not a transitive subshift \cite[Lemma 3.3]{AlcBakKon2016}. However, we do not know much about the symbolic dynamics of the transitive components of $(\vb_q, \si)$.

\begin{question}\label{q:trans}
Consider $q \in \overline{\B} \cap (q_{KL}, q_T)$. 
\begin{enumerate}[$i)$]
\item Is it true that if $(\vb_q, \si)$ is sofic, then there exists a unique transitive subshift $(X, \si)$ such that $h_{\operatorname{top}}(X) = h_{\operatorname{top}}(\vb_q)$ and $(X,\si)$ is a sofic subshift?
\item Is it true that if $q$ is a strongly $*$-irreducible number then $(\vb_q, \si)$ contains a unique transitive subshift $(X, \si)$ such that $h_{\operatorname{top}}(X) = h_{\operatorname{top}}(\vb_q)$ and $(X,\si)$ is a has the specification property?
\item What are the conditions for the quasi-greedy expansion of $q$ to ensure that $(\vb_q, \si)$ contains a unique transitive subshift $(X, \si)$ such that $h_{\operatorname{top}}(X) = h_{\operatorname{top}}(\vb_q)$ and $(X,\si)$ is synchronised? 
\end{enumerate}
\end{question}    

It would be also interesting to know if for every $q \in \vl$ the subshift $(\vb_q, \si)$ is \emph{balanced} or \emph{boundedly supermultiplicative} as in \cite[Definition 3.1, Definition 3.2]{BakGhe2015}. In particular, this question is interesting from two different perspectives: firstly, it is not clear if there is an example of a non-transitive balanced shift or a non-transitive boundedly supermultiplicative subshift. Secondly, in \cite[p. 639]{BakGhe2015} the authors ask if there is an example of a balanced subshift without the almost specification property. It would be interesting to find such example in $\Cv$ or in $\overline{\B}\cap[q_T, M+1] \setminus \Ciii$.

It would be interesting to classify the set of $q \in (\overline{\B}\cap[q_T, M+1]) \setminus \Ciii$ satisfying the weaker forms of specification defined in \cite{KwiLacOpr2016} and calculate their size. Finally, to the best of the knowledge of the author, it is not known if every transitive symmetric $q$-shift is \emph{entropy minimal}, i.e. that every subshift $(X, \sigma)$ of $(\vb_q, \sigma)$ satisfies $$h_{\operatorname{top}}(X) < h_{\operatorname{top}}(\vb_q)$$ --- see\cite{GarRaPav2019}.
   
\section*{Acknowledgements}

\noindent The author is indebted with Felipe Garc\'ia-Ramos and with Edgardo Ugalde for all their support during the development of this research. Also, the author wants to thank Simon Baker, \'Oscar Guajardo and Derong Kong for their useful remarks and comments. Finally, the author thank the anonymous referee for their rigorous and meticulous reading of this research and their very helpful suggestions that led to an improved presentation.

\end{document}